\documentclass[10pt, oneside]{article}

\usepackage{amsmath, amsfonts, amssymb, euscript, amscd, latexsym, mathrsfs, bm, color, bbm, relsize, upgreek, xfrac, tabularx, float}

\usepackage{pstricks}

 \usepackage[pdftex,colorlinks=true,
                       pdfstartview=FitV,
                       linkcolor=blue,
                       citecolor=black,
                       urlcolor=blue,
           ]{hyperref}

\usepackage{amsthm}

\newtheorem{theorem}{Theorem}[section]
\newtheorem{corollary}[theorem]{Corollary}
\newtheorem{lemma}[theorem]{Lemma}
\newtheorem{proposition}[theorem]{Proposition}

\theoremstyle{definition}
\newtheorem{definition}{Definition}[section]

\theoremstyle{remark}
\newtheorem{remark}{Remark}[section]

\newcommand{\abs}[1]{\left\vert#1\right\vert}
\newcommand{\Rnu}{\mathbb{R}}
\newcommand{\NN}{\mathbb{N}}

\def\aa#1{ \begin{align*} #1 \end{align*} }
\def\aaa#1{ \begin{align} #1 \end{align} }
\def\mm#1{ \begin{multline*} #1 \end{multline*} }
\def\mmm#1{ \begin{multline} #1 \end{multline} }

\def\begeq{\begin{equation} \begin{cases}} 
\def\endeq{ \end{cases} \end{equation}}
\def\eq#1{ \begeq #1 \endeq }
\renewcommand{\vec}[2]{\ensuremath{\begin{pmatrix} #1 \\ #2\end{pmatrix}}}

\def\bege{\begin{equation*} \begin{cases}} 
\def\ende{ \end{cases} \end{equation*}}
\def\eqq#1{ \bege #1 \ende}

\def\eqm#1{ \begin{equation*} \begin{aligned} #1 \end{aligned}\end{equation*}}

\newcommand{\eps}{\varepsilon}

\newcommand{\gt}{\geqslant}
\newcommand{\lt}{\leqslant}
\newcommand{\te}{\theta}
\newcommand{\sub}{\subset}
\newcommand{\dl}{\updelta}
\newcommand{\al}{\alpha}
\newcommand{\gm}{\gamma}

 \newcommand{\la}{\lambda}
 
 \newcommand{\sg}{\sigma}

\newcommand{\om}{\omega}
\newcommand{\mc}{\mathcal}

\newcommand{\C}{{\rm C}}
\newcommand{\td}{\tilde}

\newcommand{\cpp}{\varkappa}
\newcommand{\teta}{\upvartheta}

\newcommand{\x}{\times}
\newcommand{\mto}{\mapsto}

\newcommand{\rf}{\eqref}
\newcommand{\bi}{\begin{itemize}}
\newcommand{\ei}{\end{itemize}}

\DeclareMathOperator{\ind}{\mathbbm{1}}
\newcommand{\lap}{\Delta}

\newcommand{\lb}{\label}

\newcommand{\fdot}{\,\cdot\,}

\newcommand{\dlt}{\vartheta}

\def\Rnu{{\mathbb R}}

\def\Nnu{{\mathbb N}}

\def\ffi{\varphi}

\title{Existence, multiplicity and classification results for solutions to $k$-Hessian equations with general weights}

\author{Jo\~ao Marcos do \'O \and Justino S\'anchez \and Evelina Shamarova}

\newcommand{\Addresses}{{
  \bigskip
  \footnotesize
  
  Jo\~ao Marcos do \'O (Corresponding author), \textsc{Departamento de Matem\'atica, Universidade Federal da Para\'iba, Jo\~ao Pessoa, Brazil}\par\nopagebreak
  \textit{E-mail}: \texttt{jmbo@pq.cnpq.br}

  \medskip

Justino S\'anchez, \textsc{Departamento de Matem\'{a}ticas, Universidad de La Serena,
Avenida Cisternas 1200, La Serena, Chile}\par\nopagebreak
  \textit{E-mail}: \texttt{jsanchez@userena.cl}

  \medskip

  Evelina Shamarova, \textsc{Departamento de Matem\'atica, Universidade Federal da Para\'iba, Jo\~ao Pessoa, Brazil}\par\nopagebreak
  \textit{E-mail}: \texttt{evelina.shamarova@academico.ufpb.br}

}}

\date{}

\begin{document}

\maketitle

\vspace{-6mm}


\begin{abstract}
{The present paper is concerned} with {negative classical solutions to} a $k$-Hessian equation involving a nonlinearity with a general weight 
\begin{equation}
\lb{Eq:Ma:0}
\tag{$P$}
\begin{cases}
S_k(D^2u)= \la  \rho(\abs{x}) (1-u)^q &\mbox{in }\;\; B,\\
u=0 &\mbox{on }\partial B.
\end{cases}
\end{equation}
Here, $B$ denotes the unit ball in $\Rnu^n\!$, $n>2k$, $\la$ is a positive parameter and $q>k$ with $k\in\NN$.  The function $r\rho'(r)/\rho(r)$ satisfies very general conditions in the radial direction $r=\abs{x}$. {We show the existence, nonexistence, and multiplicity of solutions to Problem \rf{Eq:Ma:0}. The main technique used for the proofs is a phase-plane analysis related to a non-autonomous dynamical system associated to the equation in \rf{Eq:Ma:0}}. Further, using the aforementioned non-autonomous system,
we give a comprehensive characterization of $P_2$-, $P_3^+$-, $P_4^+$-solutions to the related problem
\begin{equation}
\lb{P2P3}
\tag{$\hat P$}
\begin{cases}
S_k(D^2 w)=  \rho(\abs{x}) (-w)^q,\\
w<0,  
\end{cases}
\end{equation}
given on the entire space $\Rnu^n$\!. In particular, we describe new classes of solutions: fast decay $P^+_3$-solutions and $P_4^+$-solutions.
\end{abstract}
2020 Mathematics Subject Classification. Primary: 35B33; Secondary: 34C37, 34C20, 35J62, 70K05.


\noindent {\em Keywords:}\, $k$-Hessian operator; Radially symmetric solutions; Non-autonomous system; Phase space analysis; Critical exponents; 
Singular solution; Intersection number; $P_2$\,-,  $P_3^+$-, $P_4^+$-solutions.
{\footnotesize}


\section{Introduction}
This paper aims to address the question of the existence, nonexistence and multiplicity of radially symmetric bounded solutions to  Problem \rf{Eq:Ma:0} involving a $k$-Hessian. 
Furthermore, {it aims} to characterize {$P_2$\,-, $P_3^+$-, and $P_4^+$-solutions} to Problem \rf{P2P3} (see the definition in Section \ref{tos}) {and to find
	their asymptotic behavior in a neighborhood of $+\infty$.}

{The motivation for studying problems like \rf{Eq:Ma:0} and \rf{P2P3} is twofold.  First of all, in the Laplacian case, Problem \rf{P2P3} appears in connection with the stationary case of the Vlasov-Poisson system describing stellar dynamics. More specifically, one needs to construct a stationary  spherically symmetric stellar dynamics model 
	$(\Phi, \rho, U)$, where $\Phi$ is the first integral of the stationary Vlasov-Poisson system,
	$\rho$ is a local density, and $U$ is the Newtonian potential, 
	which turns out to be a positive solution 
	to the radial version of the equation $\lap U + h_\ffi(|x|, U)=0$ (see \cite{batt86}). 
	Above, the function $\ffi$ is to be fixed a priori, and the model 
	$(\Phi, \rho, U)$ is constructed via this function $\ffi$, while
	different choices of $\ffi$ lead to different stellar dynamic models (see \cite{batt86}).
	In \cite{BattLi}, the authors considered only two particular types of stellar dynamics models
	determined by two particular choices of the function $\ffi$.
	This led to the Matukuma and Emden-Fowler equations, i.e.,
	to two particular types of our weight function $\rho$. Thus,
	even in the Laplacian case, Problem \rf{P2P3}  appears to be an interesting
	object to study as long as general weight functions $\rho$ are concerned.}
{The second goal, which is related to the presence of the $k$-Hessian operator on the left-hand sides
	of \rf{Eq:Ma:0} and \rf{P2P3}, is motivated by the extension of the results of  \cite{MiSV19} and \cite{SaVe17} to a more general class of weights.}

$k$-Hessian equations constitute an important class of fully nonlinear PDEs, so they have been studied 
	by many authors \cite{CaNS85, Dai17, OSS22, OOUb19,  MiSV19, NaSa20,
		SaVe16, SaVe17, TrXu97, TrXu99, Tso90, Wang09, WaLe19, Wei16, Wei17}. 
When restricted to so-called $k$-admissible solutions 
(see the definition in Section \ref{sec2}), these equations are elliptic, and their solutions enjoy properties similar to those of  semilinear equations with the Laplacian.

%

Following the strategy developed in \cite{SaVe17} and using ideas from \cite{Miya16},  
recently, in \cite{MiSV19}, it has been investigated the existence and multiplicity of radially symmetric bounded solutions for problem \rf{Eq:Ma:0} with a particular weight of Matukuma-type given by 
$\rho(\abs{x})=|x|^{\mu-2}(1+|x|^2)^{-\frac{\mu}{2}}$, where $\mu\gt 2$ is an additional parameter and $q>k$. 
Furthermore, in \cite{BattLi},
the authors studied $P_2$\,- and $P_3^+$-solutions
for a particular case of \rf{P2P3}, where 
the equation was in the dimension three, containing the Laplacian, and the 
source was of Matukuma-type.

In \cite{MiSV19},  where the weight $\rho$ is not a pure power as in \cite{SaVe17}, the problem was reduced to a two-dimensional non-autonomous Lotka-\!Volterra system, which was considered as an asymptotically autonomous system in the sense of  H.~Thieme \cite{Thie94}. In studying {bounded (regular)} solutions, we note that a structural property of the equations considered in \cite{SaVe16, SaVe17}, as the scaling invariance, played an important role. The similar property in \cite{MiSV19} is an asymptotic scaling invariance of the equation. On the other hand, to obtain multiplicity, a key point is to construct a singular solution and study the intersection number between regular and singular solutions, which is a common strategy for studying these relationships.

We mention that a dynamical system approach to a class of radial weighted fully 
nonlinear equations involving Pucci extremal operators was recently applied in 
\cite{Maia}. The resulting dynamics are induced by an autonomous quadratic system, obtained after a suitable transformation, similar to our change of variables. Nevertheless, their approach is quite simple, and the presence of the weight $|x|^a$ in the equation does not produce additional difficulties. However, in our case, analyzing the flow generated by the non-autonomous quadratic dynamical system is more delicate. Furthermore, our proofs do not involve any energy function obtained via Pohozaev-type identities.

  %

An important ingredient of the current study is identifying the function $R(r) ={r\rho'(r)}/{\rho(r)}$ and understanding that the existence of the finite limit $l_0 = \lim_{r\to 0} R(r)$ is responsible for the multiplicity of solutions to Problem \rf{Eq:Ma:0}. We remark that, even in light of \cite{MiSV19} and \cite{SaVe17}, the aforementioned observation is not obvious; importantly, it led to important generalizations and further developments  of the results of \cite{MiSV19} and \cite{SaVe17}.
As such, the existence and the value of the finite limit $l_\infty = \lim_{r\to +\infty} R(r)$ is responsible for the behavior of solutions to Problem \rf{P2P3} in a neighborhood of $+\infty$, {as well for their classification as a $P_2$\,-, $P_3^+$-, or $P_4^+$-solutions}. We {reinforce} that the previous results on the existence and multiplicity of solutions to Problem   \rf{Eq:Ma:0} \cite{MiSV19, SaVe17} were obtained only for weight functions of particular types, e.g., the Matukuma weight, so the
important advance of this work is describing the most general class of weights $\rho(\fdot)$ to  which the results of \cite{MiSV19} and \cite{SaVe17}  can be extended. On the other hand, identifying the function $R(r)$ and specifying its behavior as $r\to +\infty$ allows to use the associated Lotka-\!Volterra system for studying the behavior of solutions to the related Problem \rf{P2P3} at a neighborhood of $+\infty$
and for classifying respective solutions as $P_2$, $P_3^+$, or $P_4^+$. 
This classification, in particular, is valid for the Matukuma weight considered in \cite{MiSV19} and for the power weight $|x|^\sg$, considered in \cite{SaVe17},  so it complements the latter results.

In characterizing $P_3^+$- and $P_4^+$-solutions to Problem \rf{P2P3}
an important role is played  by a certain parameter $\dl$,  which, in the case of $P_3^+$- and $P_4^+$-solutions,   determines the rate of decay.
The particular case considered in \cite{BattLi} corresponds to the situation $\dl = 0$  and leads to $P_3^+$-solutions of slow log-like decay. 
{To the authors' knowledge,  $P_3^+$-solutions of fast algebraic decay, corresponding to the case $\dl>0$,   and $P_4^+$-solutions, corresponding to the case $\dl<0$, were first introduced and characterized in the present work. }

%

Thus, in this paper, we proceed further to study weighted problems for Hessian equations extending, particularly, the results obtained in \cite{BattLi, MiSV19, SaVe16, SaVe17} by considering a large class of weights. Furthermore, with respect to the work \cite{BattLi}, 
apart from more general weights, our results on $P_2$- , $P_3^+$-, and $P_4^+$-solutions  are valid for space dimensions $n\gt 3$ and numbers $k<\frac{n}2$, while in \cite{BattLi}, $n=3$ and $k=1$.

The change of variables introduced in \cite{SaVe17} (see Subsection \ref{chanofv}  and Appendix for details) reduces the equation in \rf{Eq:Ma:0} to the non-autonomous Lotka-\!Volterra system:
\begin{equation}\lb{LVSrho}
\tag{$TS_{q,\nu}$}
\begin{cases}
\frac{dx}{dt}=x\left(\nu(t)-x-q y\right),
\,\\
\frac{dy}{dt} = y\left(-\frac{n-2k}{k}+\frac{x}{k}+y\right),
\end{cases}
\end{equation}
where 
\begin{equation*}
\nu(t)=n+ R(r) \quad\text{and}\quad  t=\ln (r).
\end{equation*}
The cases $\rho(r)\equiv 1$ and $\rho(r)=r^\sigma\; (\sigma\gt 0)$ were recently studied in \cite{SaVe16} and \cite{SaVe17}, respectively. Note that both cases lead to an autonomous Lokta-\!Volterra system. Hence, it is interesting to examine what happens with the solutions when the function $R(r)$  is not constant. Note that this function appears naturally in the system \rf{LVSrho} as the definition of $\nu(t)$ shows. On the other hand, an important exponent appearing in the main results of \cite{SaVe16} and \cite{SaVe17} is 
\begin{equation}\lb{Tsoexpo}
	q^*(k,\sigma):=\frac{(n+2)k+\sigma (k+1)}{n-2k}, \;\;\sigma\gt 0.
\end{equation}
If the function $R$ is not  constant, then the \lq\lq critical exponent\rq\rq $\frac{(n+2)k+R(r) (k+1)}{n-2k}$ will vary with $r$ and the structure of the solution set will be more complex. A nonlinear model where this situation occurs was studied in \cite{MiSV19}. We point out that the results obtained here cannot be derived using the standard theory by the classical Emden-Fowler transformation.

%

In the first part of the paper, {dedicated to Problem \rf{Eq:Ma:0},} we assume that the weight function 
$\rho$ satisfies the conditions $(\rho. 1)$--$(\rho. 3)$:
\begin{itemize}
\item[$(\rho. 1)$] $\rho\in \C^2(0,\infty)\cap \C[0,\infty)$ 
with $\rho(r)>0$ for $r>0$. 
\item[$(\rho. 2)$] For the function $R(r)=r\frac{\rho'(r)}{\rho(r)}$ it holds {that}
the limit $l_0=\lim_{r\to 0}R(r)$ exists and one of the conditions, (1) or (2), is fulfilled
\bi
\item[(1)] $l_0>R(r)$ for all $r>0$ and $q\gt q^*(k,l_0)$,
\item[(2)] $l_0\gt R(r)$ for all $r>0$ and $q > q^*(k,l_0)$.
\ei
where $q^*(k,l_0)$ is defined by \rf{Tsoexpo}.
\item[$(\rho. 3)$] The function $K(r):=r^{-l_0}\rho(r)$ satisfies the condition
 $0<K(r) < L$\, for all $r>0$ and for some constant $L>0$.
 \end{itemize}
 In Remark \ref{rk1111}, we will show
the existence of the limit $\lim_{r\to 0} K(r)>0$, denoted by $K(0)$.

When the limits $\nu_{\pm}: =\lim_{t\to \pm \infty} (n+R(e^t))$ exist, we may consider the system \rf{LVSrho} as an asymptotically autonomous system in the sense of Thieme \cite{Thie94}. 
We will denote these autonomous {systems} by $(LVS_{q,\nu_{+}})$
and $(LVS_{q,\nu_-})$, respectively. Thus, we can describe the flow of \rf{LVSrho} from this autonomous systems. To do so, we follow the same approach as in \cite{MiSV19},  i.e., we use dynamical-systems tools, the intersection number between a regular and a singular solution and the method of super and subsolutions. 
{We mention that system $(LVS_{q,\nu_-})$ coincides with the
autonomous Lotka-\!Volterra system used for studying Problem \rf{Eq:Ma:0} with $\rho(\abs{x})=|x|^{l_0}$. In this case, the multiplicity of solutions is related to
the Tso and Joseph-Lundgren type exponents (see \cite{SaVe17})}. 
{On the other hand, 
two stationary} points of $(LVS_{q,\nu_-})$, denoted {by $P_4(\hat x, \hat y)$} 
and $P_3(n+l_0 , 0)$,  
are 
{relevant} to obtaining a singular solution and a bounded solution 
for the radial version of \rf{Eq:Ma:0}, denoted {by} \rf{Plambda}. 
{Namely, we prove 
that the trajectories of \rf{LVSrho} that start at the stationary point $P_3(n+l_0 , 0)$ are characterized by the existence of bounded solutions to Problem \rf{Plambda} 
(see Proposition \ref{prop35}). In turn, the trajectories of \rf{LVSrho} that start at the stationary point ${P_4}(\hat x, \hat y)$ yield the existence of a singular solution to \rf{Plambda} for some $\la>0$
(see Proposition \ref{S3L2}).
%

%

In the second part of the paper, dealing with $P_2$\,-, $P_3^+$-, and 
$P_4^+$-solutions to Problem \rf{P2P3}, we use the above assumptions 
$(\rho. 1)$, $(\rho. 2)$ and the assumptions $(\rho. 4)$, $(\rho. 5)$ below.
Also, to characterize $P_3^+$- and $P_4^+$-solutions to Problem \rf{P2P3}, 
we introduce the parameter
\aaa{
\lb{dl1111}
\dl = -\frac{2k+l_\infty}{k},
} 
whose sign determines whether a solution is $P_3^+$ or $P_4^+$; moreover, 
$\dl$ characterizes  the rate of decay. 
Assumptions $(\rho. 4)$ and $(\rho. 5)$ read as follows:
\bi
\item[$(\rho. 4)$] The limit $l_\infty=\lim_{r\to\infty}R(r)$ exists 
and one of the conditions, (1) or (2), is fulfilled
\bi
\item[(1)] $l_\infty < l_0$ and $q \gt q^*(k,l_0)$,
\item[(2)] $l_\infty \lt l_0$ and $q > q^*(k,l_0)$.
\ei
\item[$(\rho. 5)$] There exists $\teta>0$ such that $R(r) - l_\infty = O(r^{-\teta})$ as $r\to +\infty$.
\ei
{In order to obtain a sharper asymptotic representation for
the component $x(t)$ of the orbit of \rf{LVSrho} associated to
a $P^+_3$-solution, we use the additional assumption
$(\rho. 6)$ which reads as follows:}
\bi
\item[$(\rho. 6)$] One of the conditions, (1) or (2), is fulfilled:
\bi
\item[(1)] $\lim_{r\to+\infty} r^\dl (R(r) - l_\infty)\in \Rnu$ and
$n+l_\infty > \dl$;
\item[(2)] $\lim_{r\to+\infty} r^\dl |R(r) - l_\infty|=+\infty$; 
the limit 
$\hat \nu = - \lim\limits_{r\to+\infty} \frac{\ln |R(r) - l_\infty|}{\ln r}$
exists and $n+l_\infty > \hat \nu$.
\ei
\ei

%

Examples of weights satisfying the above conditions are
$\rho(\abs{x})={a|x|^l}/{(\td a+|x|^\tau )}$ with $0<\tau\lt l$ and $a, \td a >0$,
$\rho(|x|) = c |x|^\sg$, $\sg>0$, and $\rho(x) = c$, where $c>0$ is a constant.
For more examples, see Remark \ref{rm2222}.

When Problem \rf{P2P3} is concerned, 
the autonomous system $(LVS_{q,\nu_+})$ plays an important role.
More specifically,  $(LVS_{q,\nu_+})$ along with
Thieme's theorem \cite{Thie94} allows to conclude 
that the $\om$-limit set for solutions to Problem \rf{P2P3}
may consist either of point $P_2(0,\frac{n-2k}k)$ or point
$P_3^+(n+l_\infty,0)$ or point 
$P_4^+ \big(\frac{q(n-2k)- k(n+l_\infty)}{q-k}, \frac{2k+l_\infty}{q-k}\big)$. 
 These three options correspond to $P_2$-, $P_3^+$-, or $P_4^+$-solutions,
respectively. {Moreover, we show that 
the behavior of solutions as $r\to+\infty$ is determined by 
the intersection of the  associated {trajectory} $\ffi(t)$ of the Lotka-\!Volterra system \rf{LVSrho}}
with the area
\aaa{
\lb{g-1111}
{G_\mp = \{(x,y)\in\Rnu_+^2: (k+1)x + k(q+1)y \lessgtr (n-2k)(q+1)\}}
}
and by the parameter $\dl$, introduced above.
More specifically, our {characterization} can be
summarized in the following table:
\begin{table}[h]
\centering
\begin{tabularx}{\textwidth}{| X | X | X | X | X|} 
\hline
 \small $\dl>0$, $\exists\, t_0$: $\ffi(t_0)\in G_-$  & \small $\dl=0$, $\exists\, t_0$: 
 $\ffi(t_0)\in G_-$ & \small
$\dl<0$, $\exists\, t_0$: $\ffi(t_0)\in G_-$ & \small $\dl\in\Rnu$, $\ffi(t)\in G_+$ $\forall t$ \\
\hline
 \small  $P_3^+$-solution \phantom{ssssssssssss}  of fast decay 
& 
\small   
$P_3^+$-solution \phantom{ssssssssssss}  of slow decay 
&
 \small  $P_4^+$-solution
 &
\small   $P_2$-solution
\\  \hline
\end{tabularx}
\end{table}
\\
Furthermore, we establish that if the solution is 
$P_3^+$, then its rate of decay is of order $r^{-\dl}$ if $\dl>0$ (fast decay) and 
of order $(\ln r)^{-\frac{k}{q-k}}$ if $\dl = 0$ (slow decay); {if the solution
is $P_4^+$, then its rate of decay is of order  $r^{\frac{\dl k}{q-k}}$ ($\dl<0$).}
 We reinforce that in the 
situation considered in \cite{BattLi}, one has
$l_\infty = - 2$ and $k=1$, which corresponds to the case $\dl = 0$, so the work 
\cite{BattLi} only deals with 
$P_2$\,-solutions and slow decay $P_3^+$-solutions. The situation of \cite{BattLi}  is reflected
in the second and the fourth columns of the above table. The cases displayed in 
the first and the third columns are new and studied in the present work for the first time.


We indicate the contents of the individual sections. 
In Section \ref{sec2}, we briefly introduce the $k$-Hessian operator and provide necessary definitions. Moreover, 
the main results of this work are announced in this section. 
The detailed description of the main results is splitted into Sections 
\ref{sec3} and \ref{secP2P3}.
Section \ref{sec3} is dedicated to the existence, non-existence, and multiplicity of solutions to Problem \rf{Eq:Ma:0}. 
We would like to emphasize that compared to \cite{MiSV19}, 
most of the proofs of Section \ref{sec3} contain significant changes
due to the presence of the general weight $\rho$, the number $l_0$,
and the asymptotic representation $\rho\sim K(0)r^{l_0}$ as $r\to 0$.
The contents of the subsections of Section \ref{sec3} is as follows.
In Subsection \ref{Sec:Exist}, we give an existence and a non-existence result of classical solutions to Problem \rf{Eq:Ma:0}  (Theorem \ref{Exist}). In Subsection 
\ref{subsec:32}, we construct a singular solution and obtain 
its associated parameter $\la$.
In Subsection \ref{InterNumber},  we study the number of intersections points 
between a regular solution and a singular solution. Further, using scaling arguments, we study the convergence of regular solutions to singular solutions and prove Theorem \ref{S4L4} on the multiplicity of solutions to Problem \rf{Eq:Ma:0}.
Finally, Section \ref{secP2P3} is dedicated to a characterization of
$P_2$\,-, $P_3^+$-, and $P_4^+$-solutions to Problem \rf{P2P3}. 
We remark that most of the results of Section 4 do not have analogs in the existing literature.

%

\section{Preliminaries and main results}
\lb{sec2}
First, we review fundamental concepts and properties for the $k$-Hessian operators. 
For $k\in\{1,...,n\}$, let $\sigma_k:\Rnu^n\to \Rnu$ denote the $k$-th elementary symmetric function
\[
\sigma_k\left(\la\right)=\sum_{1\lt i_1<...<i_k\lt n}{\la_{i_1}\cdots\la_{i_k}},
\]
and let $\gm_k$ denote the set $\gm_k=\{\la=(\la_1,...,\la_n): \sigma_1(\la)\gt 0,...,\sigma_k(\la)\gt 0\}$. For a twice differentiable function $u$ defined on a smooth domain $\Omega\subset\Rnu^n$, the {\it $k$-Hessian operator} is defined by  
$
S_k\left(D^2u\right)=\sigma_k\left(\la\left(D^2u\right)\right),
$
where $\la\left(D^2u\right)$ are the eigenvalues of $D^2u$. Equivalently, $S_k\left(D^2u\right)$ is the sum of the $k$-th principal minors of the Hessian matrix. See e.g. \cite{Wang09, Wang94}. Two relevant examples in this family of operators are the Laplace operator $S_1\left(D^2u\right)=\lap u$ and the Monge-Amp\`{e}re operator $S_n\left(D^2u\right)=\mbox{det}\left(D^2u\right)$. They are fully nonlinear when $k\gt 2$, but not elliptic in the whole space $\C^2(\Omega)$. In order to overcome this inconvenient L. Caffarelli, L. Nirenberg and G. Spruck \cite{CaNS85} consider the class of functions 
\[
\Phi^k\left(\Omega\right)=\left\lbrace u\in \C^2\left(\Omega\right)\cap C\left(\overline{\Omega}\right): \la\left(D^2u\right)\in\gm_i,\, i=1,\ldots,k\right\rbrace .
\]
The functions in $\Phi^k\left(\Omega\right)$ are called {\it admissible} or $k$-{\it convex functions}. Further, $S_k\left(D^2u\right)$ turns to be elliptic in the class of $k$-convex functions. Denote by $\Phi_0^k\left(\Omega\right)$ the set of functions in $\Phi^k\left(\Omega\right)$ that vanish on the boundary $\partial\Omega$. An interesting property of the set $\Phi_0^k\left(\Omega\right)$ is that their elements are negative in $\Omega$. Further, the $k$-Hessian operators have a divergence structure (see, for instance, \cite{Wang94}). The study of $k$-Hessian equations has many applications in geometry, optimization theory and other related fields. See \cite{Wang09}. These operators have been studied extensively, starting with the seminal work \cite{CaNS85}. See, e.g., \cite{Jacobsen99, Jacobsen04, JaSc02, Tso89, Tso90}. Recently, this class of operators has attracted renewed interest. See e.g. \cite{Dai17, OOUb19, MiSV19, NaSa20, WaLe19, Wei16, Wei17}.

%

\medbreak

Let $\Omega=B$ be the unit ball in $\Rnu^n$. It is well known that the $k$-Hessian operator, when acting on radially symmetric $\C^2$-functions, can be written as 
\[
S_k(D^2u)=c_{n,k}\,r^{1-n}\left(r^{n-k}(u')^k \right)'=kc_{n,k}r^{1-k}(u')^{k-1}\left(u''+\frac{n-k}{k}\frac{u'}{r}\right),
\]
where $r=|x|>0$ and $c_{n,k}$ is defined by $c_{n,k}=\binom{n}{k}/n$. Here $u'$ denote the radial derivative of the radial function $u$.

Consider the problem
\begin{equation}
\lb{Plambda}
\tag{$P_{\la}$}
\begin{cases}
c_{n,k}r^{1-n}\left(r^{n-k}(u')^k \right)' = \la \rho(r) (1-u)^q, &  0<r<1,\\
u(r)  < 0, &  0\lt r<1,\\
u(1)=0. 
\end{cases}
\end{equation}
As in \cite{MiSV19}, we consider the space of functions $\Phi_0^k$ defined on {$I=(0,1)$} for Problem \rf{Plambda}:
\[
\Phi_0^k=\{u\in \C^2({I})\cap \C^1({\bar{I}}): \left(r^{n-i}(u')^i \right)'\gt 0\;\;\mbox{in}\;\; {I} ,\, i=1,...,k,\,u'(0)=u(1)=0\}.
\]
\begin{remark}\lb{negative}
Note that, if $u\in\Phi_0^k$ is a solution of \rf{Plambda} then, in particular, $S_1(D^2 u)=r^{1-n}(r^{n-1}u')'\gt 0$. Thus $G(r)=r^{n-1}u'$
is nondecreasing, since $G(0)=0$, we deduce that $G\gt 0$ and hence $u$ is nondecreasing. As a consequence, $u'\gt 0$ and $u<0$ on $[0,1)$. {Further, using the integral form of equation in \rf{Plambda} we see that $u$ is strictly increasing, and therefore, $u'(r)>0$ for $0<r\lt 1$}.
\end{remark}
Thus, negative radial solutions to Problem \rf{Eq:Ma:0} are those that solve \rf{Plambda}
and vice versa.

%

\begin{definition} 
We say that a function $u$ is:
\bi
\item[(i)] a {\it classical solution} to \rf{Plambda}  if $u\in \Phi_0^k$ and the equation
 in \rf{Plambda} holds;
\item[(ii)] an {\it integral solution} to \rf{Plambda} if $u$ is absolutely continuous on $(0,1]$, 
$\int_0^1 s^{n-1} \rho(s)(1-u(s))^qds < \infty$,
$u(1)=0$, and the equality 
\aaa{
\lb{int1111}
c_{n,k}r^{n-k}(u'(r))^k = \la\int_0^r s^{n-1} \rho(s)(1-u(s))^qds
}
holds for all $r\in I$.
\ei
\end{definition}
\begin{remark}
Note that a classical solution is always an integral solution since $u'(0) = 0$.
In this case, the equation in \rf{Plambda} is equivalent to \rf{int1111}.
Furthermore, since for all $0<r_0<r <1$,
\aa{
c_{n,k}(r^{n-k}(u'(r))^k - r^{n-k}_0(u'(r_0))^k)  
= \la\int_{r_0}^r s^{n-1} \rho(s)(1-u(s))^qds,
}
a solution to \rf{Plambda} is an integral solution if and only if
$\lim_{r\to 0} (r^{n-k}(u'(r))^k = 0$.
\end{remark}

\medbreak

Now we give the definition of super and subsolutions for \rf{Eq:Ma:0}. For the method of super and subsolutions, see \cite[Theorem 3.3]{Wang94}.
\begin{definition} 
A function $u\in\Phi^k(B):=\{u\in \C^2(B)\cap C(\overline{B}): S_{i}(D^2 u)\gt 0\;\;\mbox{in}\;B,\, i=1,...,k\}$ is called a {\it subsolution} (resp. {\it supersolution}) to
Problem \rf{Eq:Ma:0} if
\begin{equation*}
\begin{cases}
S_k(D^2u)\gt (\mbox{resp.}\lt)& \la \rho(s) (1-u)^q \;\;\mbox{in }\,\, B,\\
u\lt (\mbox{resp.}\gt)\;\; 0&\qquad\qquad\;\,\mbox{on }\; \partial B.
\end{cases}
\end{equation*}
\end{definition}

{In what follows, we will need the notion of maximal solutions.}

\begin{definition} 
\lb{def1}
We say that a function $v$ is a {\it maximal} solution to \rf{Eq:Ma:0} if $v$ is a solution to \rf{Eq:Ma:0} and, for each subsolution $u$ to \rf{Eq:Ma:0}, we have $u\lt v$.
\end{definition}
\begin{definition}
We say that $u$ is a {\it maximal} solution to \rf{Plambda}
if it is a solution to \rf{Plambda} and $v(x)=u(|x|)$ is a maximal solution to \rf{Eq:Ma:0}
in the sense of Definition \ref{def1}.
\end{definition}

%

{We shall prove the following.}
\begin{theorem}\lb{Exist}
Assume $n > 2k$, $q > k$, $l_0>2k$ {and the} conditions $(\rho. 1)$--$(\rho.3)$ hold. 
Then, there exists $\la^*>0$ such that for each $\la\in (0,\la^*)$, Problem \rf{Plambda} admits a maximal bounded solution. Moreover, there is at least one  
integral solution for $\la=\la^*$, possibly unbounded, 
and no classical solutions for all $\la>\la^*$. 
Furthermore, we have the following lower bound for $\la^*{:}$
\aaa{
\lb{boundla1111}
\la^* \gt \C^{-1}\vec{n}{k}\left(\frac{q-k}{q}\right)^q\left(\frac{2k}{q-k}\right)^k,
}
where $C=C(\rho):=\max_{r\in [0,1]} \rho(r)>0$.
\end{theorem}
{We mention that the above result gives a natural extension, valid for our general class of weights, of the corresponding results in \cite{MiSV19}.}
\medbreak

{To establish our next result on multiplicity of solutions, we need to introduce a second} relevant exponent recall that the first exponent $q^*(k,\sigma)$ is defined by \rf{Tsoexpo}, namely
\begin{equation*}\lb{Exp:critical:Intro}
q_{JL}(k,\sigma):=
\begin{cases}
k\frac{k(k+1)n-k^2(2-\sigma)+2k+\sigma-2\sqrt{k(2k+\sigma)[(k+1)n-k(2-\sigma)]}}{k(k+1)n-2k^2(k+3)-2k\sigma-2\sqrt{k(2k+\sigma)[(k+1)n-k(2-\sigma)]}}, & n>2k+8+\frac{4\sigma}{k},\\
\infty, & 2k < n \lt 2k+8+\frac{4\sigma}{k}.
\end{cases}
\end{equation*}
\noindent {This is the Joseph-Lundgren-type exponent}.
{The exponent $q_{JL}(k,\sigma)$ was found in \cite{SaVe17} in the study of the multiplicity of radial bounded solutions to \rf{Eq:Ma:0} with $\rho(\abs{x})=|x|^{\sigma}$.}
We remark that for $k>1$, $q_{JL}(k,0)$ is a critical exponent for the existence of intersection points of any two positive radial solutions to \rf{Eq:Ma:0} with $\rho\equiv 1$ (see \cite{Miya16, MiTa17}). It is also critical for the existence of non-trivial stable solutions to the $k$-Hessian equation $S_k(D^2 V)=(-V)^p$ on $\Rnu^n$ (see \cite{WaLe19}).

%

\medbreak

We are now in a position to state our result on multiplicity of solutions.
\begin{theorem}\lb{S4L4} 
Let $n > 2k$ and $q > k$.  Assume that conditions $(\rho. 1)$-$(\rho. 3)$ hold
and, moreover,  $\rho(r)$ has at most polynomial growth as $r\to\infty$.
Further assume that $q^*(k,l_0)<q < q_{JL}(k,l_0)$. Then, there exists 
$\td\la  \in (0, \la^*)$ such that for each $N\in\Nnu$, one can find an $\eps>0$ such that 
for $|\la-\td\la |<\eps$, \rf{Plambda} has at least $N$ solutions.
In particular, if $\la=\td\la $, problem \rf{Plambda} has infinitely many solutions.
\end{theorem}

{Note that the exponents $q^*(k,l_0)$ and $q_{JL}(k,l_0)$ retain their role in Problem \rf{Plambda} independently of the weight functions considered.}
On the other hand, the parameter $\td\la $ and its associate {unbounded (singular)} solution are explicit in the case of a weight of the form {$\abs{x}^\sigma$} (See \cite[Theorem 3.1 (I)]{SaVe17}). In the more general setting under consideration, the parameter 
$\td\la $ in Theorem \ref{S4L4} can be written in terms of an orbit (from which we construct a singular solution of \rf{Plambda}) and the values of {$\rho$} on the boundary of the unit ball (see Proposition \ref{S3L2} below). See Sections
\ref{subsec:32} and \ref{InterNumber} for more comments on this topic.

We finally state our results on the second problem of our interest,
Problem \rf{P2P3}. These results characterize {$P_2$\,-, $P_3^+$-, and $P_4^+$-solutions}.
First of all, we remark that as a corollary of Thieme's theorem \cite{Thie94}, in Subsection
\ref{useful}, we prove that the $\om$-limit set of the Lotka-\!Volterra system 
\rf{LVSrho} can be one of the points: $P_2(0,\frac{n-2k}k)$,
$P_3^+(n+l_\infty,0)$, or 
{$P_4^+ \big(\frac{q(n-2k)- k(n+l_\infty)}{q-k}, \frac{2k+l_\infty}{q-k}\big)$}.  Roughly speaking, $P_2$\,-, $P_3^+$-, and $P_4^+$-solutions are defined
by the $\om$-limit sets of the associated orbits of the non-autonomous Lotka-\!Volterra
system \rf{LVSrho}.
\begin{theorem}[On $P_2$\,-solutions]
\lb{p2222}
Let $(\rho. 1)$, $(\rho. 2)$,  $(\rho. 4)$, and $(\rho. 5)$ hold. Then,
the following conditions are equivalent:
\bi
\item[(i)] $w$ is a $P_2$\,-solution.
\item[(ii)] 
If $\ffi(t) = (x(t),y(t))$ is the associated orbit of the non-autonomous Lotka-\!Volterra system \rf{LVSrho}, then
there exist constants $c_1>0$ and 
$c_2= c_1 \big(\frac{k^2\gm}{n-2k}- k\big)$ such that
\aaa{
\lb{p2xy}
x(t) = c_1  e^{-\gm t} (1+o(1)), \quad 
y(t) = \frac{n-2k}{k} + c_2  e^{-\gm t}  (1+o(1)) \quad (t\to +\infty),
}
where $\gm =  \frac{q}{k}(n-2k) - (n+l_\infty)>0$.
\item[(iii)] $w(r) = - c_3 r^{-\frac{n-2k}{k}}(1+o(1))$, \quad
$w'(r) = c_4 r^{-\frac{n-k}{k}}(1+o(1))\; \; {(r\to +\infty)}$,\\
where $c_3= \big(c_{n,k} c_1c_\rho^{-1}\big)^{\frac1{q-k}} 
\big(\frac{n-2k}k\big)^{\frac{k}{q-k}}$,
\; $c_4= \big(c_{n,k} c_1c_\rho^{-1}\big)^{\frac1{q-k}} 
\big(\frac{n-2k}k\big)^{\frac{q}{q-k}}$.
\item[(iv)] $\ffi(t) \in G_+$ for all $t\in \Rnu$.
\ei
Moreover, in a neighborhood of $P_2$, $y$ can be represented as
a function of $x$, which we denote by $\hat y(x)$, and
 \aaa{
\lb{haty2222}
\lim_{x\to 0} \hat y'(x) = \hat y'(0) = -\frac{n-2k}{k^2\gm + k(n-2k)}.
}
\end{theorem}
In the above item {\it (iii)}, the constant $c_\rho = \lim_{r\to\infty} \rho(r) r^{-l_\infty}$
(see Lemma \ref{krho}).

%

{Finally, in Theorems \ref{t1111} and \ref{bl1111} below, $W_-$ is defined similar to $G_-$
(see \rf{g-1111}), but via a different linear function.
 Remark that all the above-mentioned subsets 
are rigorously defined in Subsection \ref{TSqnu}.}
Also define $\zeta(t) = \nu(t)-n-l_{\infty}$.

\begin{theorem}[On $P_3^+$\!-solutions of algebraic fast decay]
\lb{t1111}
Assume $(\rho. 1)$, $(\rho. 2)$, $(\rho. 4)$, $(\rho. 5)$ and let
$l_\infty<-2k$ $(\dl>0)$. Then, the following conditions are equivalent:
\bi
\item[(i)] $w$ is a $P_3^+$-solution.
\item[(ii)]
If $\ffi(t) = (x(t),y(t))$ is the associated orbit of the non-autonomous Lotka-\!Volterra system \rf{LVSrho}, then
there exists a constant $c>0$
such that
\aaa{
\lb{yassi}
x(t) = \nu_+(1+o(1)), \qquad
y(t) = c \, e^{-\dl t} (1+o(1)) \quad (t\to +\infty).
}
\item[(iii)] It holds that
\aa{
 w(r) = -c_1\Big(1+ \frac{c}{\dl}\, r^{-\dl}\Big) + o(r^{-\dl}); \quad
w'(r) = c_2\, r^{-(\dl + 1)}(1+o(1)) \; \; (r\to +\infty),
} 
where 
$c_1 = (\nu_+c_{n,k} c^k c_\rho^{-1})^{\frac1{q-k}}$\!, \;
$c_2 = (\nu_+c_{n,k} c^q c_\rho^{-1})^{\frac1{q-k}}$\!.
\item[(iv)] There exists $t_0\in\Rnu$ such that $\ffi(t_0)\in G_-$.
\item[(v)] There exists $t_0\in\Rnu$ such that $\ffi(t_0)\in W_-$.
\ei
If, in addition, $(\rho.6)$ holds, then we have a more precise
asymptotic representation for $x(t)$ as  $t\to+\infty$:
\aaa{
\lb{x1111}
x(t) = 
\begin{cases}
\nu_+ +\frac{(\cpp-qc)\nu_+}{\nu_+ - \dl}\, e^{-\dl t} (1+o(1))  
\quad \text{under $(\rho. 6)$-(1);}\\
\nu_+ +\frac{\nu_+}{\nu_+ - \hat \nu}\, \zeta(t) (1+o(1))
\quad \text{under $(\rho. 6)$-(2),}
\end{cases}
}
where $\cpp= \lim_{t\to+\infty} e^{\dl t}\zeta(t)$
and  $\hat \nu =  \lim_{t\to+\infty}(-\frac{\zeta'(t)}{\zeta(t)})$. 
Moreover, in a neighborhood of $P_3^+$, under $(\rho.6)$-(1),  $x$ can be represented as
a function of $y$, which we denote by $\hat x(y)$; under $(\rho.6)$-(2),
$y$ can be represented as a function of $x$, denoted $\hat y(x)$, and
\aaa{
\lb{haty1111}
\begin{cases}
\hat x'(0) = \frac{(\frac{\cpp}{c}-q)\nu_+}{\nu_+-\dl}   \quad \text{under $(\rho. 6)$-(1);}\\
\hat y'(\nu_+) = 0  \quad \text{under $(\rho. 6)$-(2).}
\end{cases}
}
\end{theorem}

%

\begin{theorem}[On $P_3^+$\!-solutions of log-like slow decay]
\lb{bl1111}
Assume $(\rho. 1)$, $(\rho. 2)$, $(\rho. 4)$,  $(\rho. 5)$
and let $l_\infty = -2k$ $(\dl=0)$. Then, the following conditions are equivalent:
\bi
\item[(i)] $w(r)$ is {a} $P_3^+$-solution.
\item[(ii)]
If $\ffi(t) = (x(t),y(t))$ is the associated orbit of \rf{LVSrho}, then
\aa{
 x(t) = n-2k - \frac{qk}{q-k} \frac1t(1+o(1)), \qquad
 y(t) =\frac{k}{q-k} \frac1t (1+o(1)) \quad {(t\to +\infty)},
 }
\item[(iii)] It holds that {as $r\to +\infty$,}
\aa{
w(r) = - c_3 (\ln r)^{-\frac{k}{q-k}}(1+o(1)),\quad
w'(r) = c_4 \, r^{-1} (\ln r)^{-\frac{q}{q-k}}(1+o(1)).
} 
where $c_3 = \big( c_\rho^{-1} c_{n,k} \big(\frac{k}{q-k}\big)^k (n-2k)\big)^\frac1{q-k}$,
\, $c_4 = \big( c_\rho^{-1}c_{n,k} \big(\frac{k}{q-k}\big)^q (n-2k)\big)^\frac1{q-k}$.
\item[(iv)] There exists $t_0\in\Rnu$ such that $\ffi(t_0)\in G_-$.
\item[(v)] There exists $t_0\in\Rnu$ such that $\ffi(t_0)\in W_-$.
\ei
Moreover, in a neighborhood of $P_3^+$\!, $y$ can be represented as
a function of $x$, which we still denote by $\hat y(x)$, and
\aaa{
\lb{yn2k1111}
\lim_{x\to n-2k} \hat y'(x) = \hat y'(n-2k) = -\frac1{q}.
}
\end{theorem}
Finally, if $l_\infty > -2k$, we announce our result on $P_4^+$-solutions.
\begin{theorem}
\lb{7777}
Assume $(\rho.1)$, $(\rho.2)$, $(\rho.4)$, $(\rho.5)$ and let $l_\infty>-2k$ $(\dl<0)$. 
Then, the following conditions are {equivalent}:
\bi
\item[(i)] $w(r)$ is a $P_4^+$-solution.
\item[(ii)] It holds that {as $r\to +\infty$,}
{
\aa{
w(r) = c_3\, r^{\frac{\dl k}{q-k}}(1+o(1)), \qquad
w'(r) = - c_4 \, r^{\frac{\dl k}{q-k}-1}(1+o(1)),
}}
where 
$c_3 = (c_{n,k} c_\rho^{-1} {\td{x}}{\td{y}}^k)^{\frac1{q-k}}$, 
$c_4 = (c_{n,k} c_\rho^{-1} {\td{x}}{\td{y}}^q)^{\frac1{q-k}}$,
${P_4^+}({\td{x}},{\td{y}}) = \big(\frac{q(n-2k)- k(n+l_\infty)}{q-k}, \frac{2k+l_\infty}{q-k}\big)$.
\item[(iii)] There exists $t_0\in \Rnu$ such that $\ffi(t_0)\in G_-$.
\ei
{If, additionally, 
 \aaa{
 \lb{d1111}
l_\infty < \frac{q(n-2k) - k(n+2\mu_2)}{k+\mu_2}
 \quad \text{or} \quad 
 l_\infty > \frac{q(n-2k) - k(n+2\mu_1)}{k+\mu_1}
 }
where $\mu_{1,2} = \frac{2q}{k} - 1\mp 2\sqrt{(\frac{q}{k})^2 - \frac{q}{k}}$ 
$(\mu_1<\mu_2)$,}
then items (i), (ii), (iii) are {equivalent} to the following:
\bi
\item[(iv)] There exist constants $c_1,c_2\in \Rnu$ such that 
\aa{
x(t) ={\td{x}}+  e^{\la_1 t}(c_1+o(1)), \qquad  y(t) = {\td{y}} + e^{\la_1 t}(c_2+o(1)) \; \; {(t\to +\infty)},
}
where $\la_1<0$ is the maximal eigenvalue of the linearized 
system \rf{limit1111} at the stationary point 
${P_4^+}({\td{x}},{\td{y}})=\big(\frac{q(n-2k)- k(n+l_\infty)}{q-k}, \frac{2k+l_\infty}{q-k}\big)$.
\ei
\end{theorem}

\begin{remark}
\lb{rk1111}
Note that under $(\rho. 3)$, the positive limit $K(0)=\lim_{r\to 0}r^{-l_0}\rho(r)$
exists. Indeed, 
\aaa{
\lb{Kr1111}
(\ln K(r))' = -\frac{l_0}{r} + \frac{\rho'(r)}{\rho(r)} 
= \frac{R(r) - l_0}{r}.
}
Integrating from $\frac1n$ to r, we obtain that
\aa{
0\lt \int_0^r \ind_{[\frac1n,r]}  \frac{l_0-R(s)}{s} ds = 
\ln K\Big(\frac1{n}\Big) - \ln K(r) \lt\ln L -\ln K(r). 
}
By the monotone convergence theorem, the finite limit 
\aa{
\lim_{n\to +\infty} \int_{\frac1n}^r \frac{l_0-R(t)}{t} dt = \int_0^r  \frac{l_0-R(t)}{t} dt. 
}
exists. Therefore, the finite limit $\lim_{n\to +\infty} \ln K(\frac1n)$ also exists. 
This implies the existence of a positive limit 
$\lim_{n\to+\infty} K(\frac1n)  = K(0)$.
\end{remark}

\begin{remark}
\lb{rm2222}
Formula \rf{Kr1111} along with the argument
in Remark \ref{rk1111} imply that a general construction of weights $\rho(r)$,
satisfying $(\rho. 1)$--$(\rho. 3)$, can be done by the formula
\aaa{
\lb{rho1111}
\rho(r) = r^{l_0} K(r) = K(0) r^{l_0} \exp\Big\{\int_0^r \frac{R(s) - l_0}{s} ds\Big\},
}
where $R(\fdot)$ is to be chosen as in assumption $(\rho. 2)$.
\end{remark}

%

 \section{Existence and multiplicity results for Problem \rf{Eq:Ma:0}}
 \lb{sec3}

\subsection{Existence and non-existence of solutions of Problem \rf{Plambda}}
\lb{Sec:Exist}
In {the following}, we prove {an 
existence result for }classical solutions of Problem 
\rf{Plambda}.
\begin{lemma}\lb{max:sol}
Let $n > 2k$, $q > k$, and $l_0>2k$. 
Assume that conditions $(\rho. 1)$--$(\rho. 3)$ hold. 
Further assume there exists a negative subsolution $w\in\C^2([0,1))$ 
to Problem \rf{Plambda}.
Then, 
there exists a  classical maximal bounded solution to 
Problem \rf{Plambda}.
{Furthermore}, the classical maximal bounded solutions form a decreasing 
sequence as $\la$ increases.
\end{lemma}
\begin{proof}
 Let $M_0 = \sup_B |w|$ and
 let $\zeta(y)$ be a smooth cutting function with the support $(-M_0 -\frac12, \frac12)$
 such that $\zeta(y) = 1$ for $y\in [-M_0,0]$.
 $\zeta(\fdot)$ could be a mollified indicator function for 
 $\ind_{[-M_0-\frac14,\frac14]}(\fdot)$.
 Then, $w$ is also a subsolution to the problem obtained from \rf{Plambda}
 by replacing the right-hand side of the equation in \rf{Plambda} with  
 $\td f(x,u) = \la \rho(|x|) (1-u)^q \zeta^k(u)$.  
 We will use the notation $(\td P_\la)$ for the 
 above-described problem.
  In what follows, we will make use of Theorem 3.3 in 
\cite{Wang94}. Introducing $\td f(x,u)$ as above, we are able 
 to verify the following conditions required in this theorem:
\aa{
D_x(\td f^\frac1k) \in \C^1(\bar B \x \Rnu), \quad D^2_u(\td f^\frac1k)\gt -C_0 >-\infty.
}
Indeed, for $r=|x|$, we have
$D_{x_ix_j}\rho(r)^\frac1k = \big(D^2_r - \frac{D_r}{r}\big)  \rho(r)^\frac1k
 \frac{x_ix_j}{r^2} +\frac{D_r}{r}  \rho(r)^\frac1k  \dl_{ij}$.
Since  $\frac{\rho'(r)}{\rho(r)}\sim \frac{l_0}{r}$ as $r\to 0$, 
by L'Hopital's rule,
\aa{
l_0 = \lim_{r\to 0}\frac{r}{\sfrac{\rho(r)}{\rho'(r)}}
= \lim_{r\to 0}\frac1{1-\frac{\rho(r)\rho''(r)}{\rho'(r)^2}}
=  \lim_{r\to 0}\frac1{1- \frac1{l_0^2}\frac{r^2\rho''(r)}{\rho(r)}}.
}
This implies that $\frac{\rho''(r)}{\rho(r)}\sim \frac{c}{r^2}$ ($r\to 0$) for 
some constant $c$. Next, a straightforward computation shows that
$\frac{D_r}{r}  \rho(r)^\frac1k  \sim c_1 \rho(r)^\frac1k \frac1{r^2}$ and
$D^2_r  \rho(r)^\frac1k \sim  c_2\rho(r)^\frac1k \frac1{r^2}$ ($r\to 0$)
for some constants $c_1$ and $c_2$. Note that
under $(\rho.1)$,  $\rho(|x|)^\frac1k$  is twice continuously 
differentiable at every point $x\ne 0$.
Thus, to ensure that $\rho(|x|)^\frac1k\in \C^2$, 
it suffices to notice that $\rho(r)^\frac1k \sim K(0) r^{\frac{l_0}{k}}$ ($r\to 0$).
In this case,
$\lim_{x\to 0} D^2_x \rho(|x|)^\frac1k = 0$ if $l_0>2k$.

Since that $w$ is a subsolution
to Problem $(\td P_\la)$ and $0$ is a supersolution to  the same problem,
 by  \cite[Theorem 3.3]{Wang94}, 
there exists a classical solution $u$ to Problem \rf{Eq:Ma:0} such that
 $w\lt u \lt 0$ (note that $\zeta(u)=1$).
 Next, we define $u_i$, $i=1,2,\ldots$, as the radial classical solution of
 \eq{
\lb{it1111}
S_k(D^2 u_i)= \la \rho(|x|)(1-u_{i-1})^q &\mbox{in }\;\; B,\\
u_i=0 &\mbox{on }\; \partial B.
}
where $u_0 = 0$.
By Remark  \ref{negative}, $u_i<0$ in $B$. 
Let us show that $u_i$, $i=1,2,\ldots$, form a decreasing sequence.
By the comparison principle (see, e.g. \cite{TrXu97}),
$u_1\gt u_2$. Suppose, as the induction hypothesis, that $u_{i-1}\gt u_i$.
Then, $S_k(D^2 u_{i+1}) \gt \la \rho(|x|)(1-u_{i-1})^q$. By the {same} principle,
$u_i\gt u_{i+1}$. Furthermore, we note that $u_i\gt u$ for all $i$. Indeed, the inequality
$u_1\gt u$ follows from the comparison principle. Suppose, 
as the induction hypothesis, that $u_{i-1}\gt u$. Then,
$S_k(D^2 u_i) \lt \la \rho(|x|)(1-u)^q$. Therefore, $u_i\gt u$
by {comparison}. Thus, there exists a function,
we denote it by $u_{max}$, which is the pointwise limit of $u_i$ as $i\to+\infty$.
Note that $u_{max}$ is a radial function. Moreover, it is a solution
to \rf{Plambda}. Indeed, Problem \rf{it1111}, can be reduced to an integral
form. Passing to the limit, we obtain that $u_{max}$ satisfies the 
integral equation equivalent to \rf{Plambda}.

{It remains to show that $u_{max}(|x|)$ is the maximal solution to \rf{Eq:Ma:0}.
Let $v$ be a subsolution to \rf{Eq:Ma:0}. 
Since the identical zero is a supersolution to \rf{Eq:Ma:0},
then, by \cite[Theorem 3.3]{Wang94},  $v\lt 0$. Then, $v$ is a subsoluton
to \rf{it1111} for $i=1$; therefore, $u_1\gt v$. By the same argument as above
(i.e., involving $v$ instead of $u$), we obtain that $u_i\gt v$ for all $i$. This implies that 
$u_{max}\gt v$.}
Now let $\la_1<\la_2$ and $u_{\la_1}$, $u_{\la_2}$ be maximal solutions of $(P_{\la_i})$ ($i=1,2$), respectively. Since $u_{\la_2}$ is a subsolution of $(P_{\la_1})$, we have $u_{\la_2} \lt u_{\la_1}$ by the maximality of $u_{\la_1}$.
  \end{proof}

%

\subsubsection{Proof of Theorem \ref{Exist}}

\begin{proof}
Let $B_2$ be the ball of radius $2$ centered at zero and the constant $C>0$ be
as defined in the statement of the theorem. 
It is straightforward to verify that 
{$\phi(x) = \frac12\big(\frac{C}{nc_{n,k} }\big)^\frac1k  (|x|^2 -4)$}
is a radial classical solution to
\eqq{
S_k(D^2\phi)= C &\mbox{in }\;\; B_2,\\
\phi=0 &\mbox{on }\; \partial B_2.
}
Let $\gm<0$ be a constant such that $\phi<\gm<0$ on $\partial{B}$.
Set $M= {2\big(\frac{C}{nc_{n,k} }\big)^\frac1k}=
\max_{x\in\overline{B}}\,|\phi(x)|$ and take $\la<(1+M)^{-q}$. 
Then, since $0\lt 1-\phi \lt 1+M$ and $\rho(|x|)\lt C$, for $x\in B$,
it holds that
\aa{
S_k(D^2\phi)=C \gt\la \rho({\abs{x}})(1-\phi)^q.
}
Since $\phi$ is a subsolution to {Problem} \rf{Plambda} and 
the identical zero is a supersolution, by Lemma \ref{max:sol},
for every $\la\in (0,(1+M)^{-q})$, there exists a 
solution $u_\la\in \C^2([0,1))$ to \rf{Plambda} such that $\phi\lt u_\la \lt 0$. 
Define
\aaa{
\lb{lambda-ast}
\la^\ast=\sup\{\la>0:\;\; \text{there exists a solution} \;\; u_\la\in \C^2([0,1)) \text{ to } \rf{Plambda}\}.
}
It is obvious that $\la^\ast>0$.
To show that $\la^\ast$ is finite, we consider the Newton inequality
\aaa{
\lb{Eq:Nineq}
\lap u\gt C(n,k)[S_k(D^2 u)]^\frac{1}{k},
}
where $C(n,k)>0$ is a constant and $u\in \Phi^k(B)$, see e.g. \cite[Proposition 2.2, part (4)]{Wang94}. 
We also need to consider the weighted eigenvalue problem
\eq{
\lb{E_w}
\tag{$E_w$}
-\lap u= \la w(x)u &\mbox{in }\;\; B,\\
u=0 &\mbox{on }\; \partial B,
}
where $w(x):=\rho(|x|)^\frac{1}{k}\in L^\infty(B)$. %
Let $\la_{1,w}>0$ be the first eigenvalue and $\phi_{1,w}>0$ their associated eigenfunction for problem \rf{E_w}, see, e.g., \cite[Theorem 0.6]{AmPr93}. It is easy to see that there exists a constant $L>0$ such that for every $u<0$, we have $(1-u)^\frac{q}{k}\gt \frac{L|u|}{C(n,k)}$.


Let $\la\in (0,\la^\ast)$ and let $u_\la\in \C^2([0,1))$ be a solution to \rf{Plambda}. 
Remark that $u_\la\lt 0$.
Then, by \rf{Eq:Nineq} and the integration-by-parts formula, 
\aa{
\la_{1,w}\int_B (-u_\la) w(x)\phi_{1,w} dx = \int_B (-u_\la) \lap \phi_{1,w} dx
\gt  L\la^\frac{1}{k}\int_B (-u_\la) w(x)\phi_{1,w} dx,
}
which implies that $\la>0$ is bounded from above, and hence,
$\la^\ast$ is finite.

Now, let $\la\in (0,\la^*)$. 
Take ${\delta}\in (0, \la^* - \la)$ and note that $u_{\la+{\delta}}$ is a subsolution for \rf{Plambda}.
By Lemma \ref{max:sol}, 
$u_\la\in \C^2([0,1))$ is a maximal bounded solution to \rf{Plambda}. 
By the same lemma, $u_\la$
decreases as $\la$ increases. Define $u_{\la^*} = \lim_{\la\to\la^*} u_\la$. Rewriting
\rf{Plambda} in the integral form, we obtain
\aaa{
\lb{integral1111}
- u_\la(r)= K_\la\int_{r}^{1}\tau^{\frac{k-n}k}\Big(\int_{0}^{\tau}s^{n-1}\rho(s)(1-u_\la(s))^q\,ds\Big)^{\frac{1}{k}}d\tau,
}
where $K_\la = c_{n,k}^{-\frac1k} \la^\frac1k$.
Let us show that in the above identity, we can pass to the limit as $\la\to\la^*$
on any compact interval $[\eps, 1] \sub (0,1]$. 

To prove this, we show first the boundedness of the family $u_\la$,
uniformly in $\la\in (0,\la^*)$, on  $[\eps, 1]$.
It is known that the change of variable $u(r) = -u_\la(r \la^{-\frac1{2k}})$
(see, e.g., \cite{OSS22}) reduces {Problem} \rf{Plambda} to 
 \eq{
 \lb{G1111}
  -S_k(D^2 u) = \rho_1(r) (1+u)^q, \\
  u(0) = a,
 }
 where $a = -u_\la(0)$ and $\rho_1(r) =  \rho(r \la^{-\frac1{2k}})$.
 Let $u(r,a)$ be the $\C^2([0,1))$-solution to \rf{G1111} (satisfying $u'(0) = 0$)
 and let {$D\gt 0$}. Define $R(D,a) = u^{-1}(\fdot,a)$.
Let us show that $R(D,a)$ is well-defined for $a> D$, that is,  
a zero of $u(r,{a)}-D$ exists. 
Suppose this is not the case. Since $u(r,a)$ is decreasing, 
then $u(r,a) > D$ for all $r\gt 0$.
From \rf{G1111}, we obtain that
  $u(r,a) \lt 
 a - \frac12(1+D)^\frac{q}{k} n^{-\frac1{k}}  r^2 \to -\infty$, as $r\to +\infty$.
 This is a contradiction. Hence, $R(D,a)$ is well-defined.
Let us show that  the function $[D,+\infty)\to \Rnu$, $a\mto R(D,a)$ 
is bounded from above. 
Rescale the equation in \rf{G1111},  by setting $\mu = R(D,a)^{2k}$, as follows:
\aa{
 -S_k\big(D^2 u(\mu^{\frac1{2k}}r)\big) =  \mu\, \rho_1(\mu^{\frac1{2k}} r) (1+u(\mu^{\frac1{2k}} r))^q.
 }
Rewriting it in the integral form, we obtain
  \mm{
u(\mu^{\frac1{2k}} r) =  D+R(D,a)^2  \int_r^1  t^{-\frac{n-k}{k}}
\Big(\int_0^t s^{n-1}\rho_1(\mu^{\frac1{2k}} s) 
\big(1+u(\mu^{\frac1{2k}} s)\big)^q ds\Big)^{\frac1k} dt\\
\gt  D+ R(D,a)^2  \int_r^1  t^{-\frac{n-k}{k}} dt\, 
\Big(\int_0^r s^{n-1}\rho_1(\mu^{\frac1{2k}} s) \big(1+u(\mu^{\frac1{2k}} r)\big)^q ds\Big)^{\frac1k}. 
}
Next, by $(\rho.3)$ and Remark  \ref{rk1111}, 
$\rho_1(\mu^{\frac1{2k}} s) =  \rho(\la^{-\frac1{2k}}\mu^{\frac1{2k}} s)
\gt L_1 \mu^{\frac{l_0}{2k}} s^{l_0}$ for some constant $L_1$.
Hence, 
 \aa{
 \sup_{u\gt D} \big\{(u-D) \, (1+u)^{-\frac{q}{k}}\big\} \gt (u(\mu^{\frac1{2k}} r) - D) 
\big(1+u(\mu^{\frac1{2k}} r)\big)^\frac{q}{k} 
 \gt L_2 R(D,a)^{2+\frac{l_0}{k}} (r^\frac{l_0+2k}k - r^\frac{n+l_0}k),
 }
 where $L_2>0$ is a constant.
 Evaluating the right-hand side at e.g. $r=\frac12$,
 we obtain an upper bound $\bar R(D)$ for $R(D,a)$ whenever $a\gt D$.
 It is straightforward  to verify that the supremum on the left-hand side
is $(1+D)^{1-\frac{q}{k}}$ up to a multiplicative constant. Therefore,
$\lim_{D\to+\infty} \bar R(D) = 0$.

Note that $u(\la^\frac1{2k},a) = 0$ for all $a > 0$. Take $\eps>0$.
Let $D>0$ be such that $\bar R(D) <  \la^\frac1{2k}\eps$. Then, 
on $[ \la^\frac1{2k}\eps, \la^\frac1{2k}]$, $u(\fdot, a) < D$ for $a>D$. 
If $a\lt D$, then clearly, $u(\fdot, a) \lt a \lt D$. The family 
$u(\la^\frac1{2k},a)$, parametrized by $a$, is therefore 
uniformly bounded on  $[ \la^\frac1{2k}\eps, \la^\frac1{2k}]$
for each fixed $\la\in (0,\la^*)$.
Consequently, the family $u_\la$, parametrized by $\la\in (0,\la^*)$,
 is uniformly bounded on $[\eps,1]$. 
 Therefore, for $r\in [\eps,1]$ and for $N\in \Nnu$,
 \mm{
 \lim_{\la\to\la^*} K_\la\int_r^1\tau^{\frac{k-n}k}
 \Big(\int_{\frac1N}^\tau s^{n-1}\rho(s)(1-u_\la(s))^q\,ds\Big)^{\frac{1}{k}}d\tau\\
 =K_{\la^*}\int_r^1\tau^{\frac{k-n}k}
 \Big(\int_{\frac1N}^\tau s^{n-1}\rho(s)(1-u_{\la^*}(s))^q\,ds\Big)^{\frac{1}{k}}d\tau
 \lt \lim_{\la\to\la^*} -u_\la(r) = -u_{\la^*}(r) \lt D.
 }
 Indeed, the passage to the limit is possible by the monotone convergence
 theorem and the {uniform boundedness in $\la$} of $-u_\la$ on $[\frac1N,1]$.
 Now let $N\to\infty$ and apply the monotone convergence theorem once again.
  The limit function $\tau^{\frac{k-n}k}
 \big(\int_0^\tau s^{n-1}\rho(s)(1-u_{\la^*}(s))^q\,ds\big)^{\frac{1}{k}}$ is then integrable 
 in $\tau$ on $[r,1]$. Therefore, $s^{n-1}\rho(s)(1-u_{\la^*}(s))^q$ 
 is integrable on $[0,\tau]$. This allows to pass to the limit in \rf{integral1111}
 as $\la\to \la^*$ by the dominated convergence theorem.
In addition, \rf{integral1111} implies that {$u_{\la^*}$} is an integral solution to 
$(P_{\la^*})$.

It remains to obtain a uniform lower bound for $\la^*$. For this, 
consider the function $v(r)=\frac{k}{q-k}\left(r^2-1\right)$ and
note that $v(|x|)<0$ on $B$ and
\mm{
S_k(D^2 v)=nc_{n,k}(2k)^k(q-k)^{-k}\gt \vec{n}{k}(2k)^k\frac{(q-k)^{q-k}}{q^q}(1-v)^q \\ \gt C^{-1}\vec{n}{k}(2k)^k\frac{(q-k)^{q-k}}{q^q}\rho(\abs{x})(1-v)^q.
}
Hence, $v(r)$ is a $\C^2([0,1))$-subsolution to  \rf{Plambda}  for all $\la\lt C^{-1}\binom{n}{k}(2k)^k\frac{(q-k)^{q-k}}{q^q}$. By Lemma \ref{max:sol}, there exists 
a maximal bounded classical solution to \rf{Plambda}. This
implies \rf{boundla1111} and concludes the proof.
\end{proof}

\subsection{Singular solution to \rf{Plambda}}
\lb{subsec:32}
The main part of this subsection consists of the construction of a singular solution, which is deduced directly from the dynamical system approach. 
For the reader convenience, we reintroduce some elements of this approach 
exposed in subsections 4.1, 4.2, 4.3
from \cite{MiSV19}, such as  the Lotka-\!Volterra system, the stationary points 
of $(LVS_{q,\nu_-})$, and some characteristics of the flow of \rf{LVSrho}.

\subsubsection{The new variables and the Lotka-\!Volterra system}\lb{chanofv}

Rewrite  {Problem} \rf{Plambda} as follows:
\eq{
\lb{RaPr:1}
\left(r^{n-k}(u')^k\right)'= c_{n,k}^{-1}\,r^{n-1}f(r,u),& 0<r<1,\\
u(r)<0, &0\lt r<1\\
u(1) =0,
}
where $f(r, u): =\la \rho(r)(1-u)^q$.
Let $u$ be a solution of \rf{RaPr:1} and set $w=u-1$. Then $w$ solves the equation
\aaa{
\lb{Eq:IVP:0}
\left(r^{n-k}(w')^k\right)'= r^{n-1}c_{n,k}^{-1}\,f(r,w+1),  \quad (r>0).
}
The following variables transform this equation into a Lotka-Volterra system:
\aaa{
\lb{newtrans0}
x(t)=r^k\frac{c_{n,k}^{-1}\,f(r,w+1)}{(w')^{k}},\; y(t)=r\frac{w'}{-w}, \;\, t=\ln(r),
}
whenever $r > 0$ is such that $w(r)\neq 0$ and $w'(r)\neq 0$.
 We consider the phase plane $(x,y)\in\Rnu^2$. Since we are studying negative
solutions, the points $(x(t),y(t))$ belong to the first quadrant when $w'>0$ (see Remark \ref{negative}).
We also remark that for $f(r,w+1)=\la \rho(r)(-w)^q$,
this change of variables
becomes optimal for equation \rf{Eq:IVP:0} since, depending on the weight $\rho$, it leads to either an autonomous or a non-autonomous Lotka-\!Volterra system. In fact, the variables in \rf{newtrans0} transform equation \rf{Eq:IVP:0} into the following quadratic non-autonomous system of differential equations of Lotka-\!Volterra-type
\eqq{
\frac{dx}{dt}=x\left(\nu(t)-x-q y\right),
\,\\
\frac{dy}{dt} = y\left(-\frac{n-2k}{k}+\frac{x}{k}+y\right)
}
(previously denoted by \rf{LVSrho}),
where $\nu(t)=n+r\frac{\rho'(r)}{\rho(r)}$\; and\; $t=\ln (r)$. In the  Appendix,
we provide a detailed derivation of the  Lotka-\!Volterra system \rf{LVSrho}.
Note that we can recover the function $w$ by the formula
\aaa{
\lb{inverse0}
w(r)=-\left(c_{n,k}^{-1}\la\,r^{2k}\rho(r)\right)^{-\frac{1}{q-k}}(x(t)y(t)^k)^{\frac{1}{q-k}},\; \mbox{where}\; \; r=e^t.
}
More specifically, if $(x(t),y(t))$ is a solution to  \rf{LVSrho},
it is straighforward to verify that $w(r)$, given by \rf{inverse0}, solves
\rf{Eq:IVP:0}.

Since the limit $\lim_{t\to - \infty} \nu(t) =\lim_{t\to - \infty}\left(n+\frac{\rho'(e^t)}{\rho(e^t)}\right)=:\nu_{-}$ exist, we may consider the system \rf{LVSrho} as an asymptotically autonomous system in the sense of H. Thieme (see \cite{Thie94}). Thus we can describe the flow of \rf{LVSrho} from the limit autonomous {system}
$(LVS_{q,\nu_{-}})$.

\begin{remark}
In \rf{LVSrho}, $TS$ stands for ``Thieme system''.
\end{remark}

\subsubsection{Stationary points and local analysis}
\lb{LinSPoints}
Since in the first part of the paper we study radially symmetric bounded solutions, we only focus on the autonomous system $(LVS_{q,\nu_{-}})$. In this case the decomposition
$\nu(t)=n+l_0-\kappa(t)$ is useful when $t\to -\infty$, where the function $\kappa$ is such that $\lim_{t\to -\infty}\abs{\kappa(t)}=0$.

Let $P=(a,b)$ be a stationary point of $(LVS_{q,\nu_{-}})$. {Introducing the coordinates 
$\bar x:=x-a,\, \bar y:=y-b$ and using the  above decomposition we can write \rf{LVSrho} as a time-dependent perturbation of $(LVS_{q,\nu_{-}})$}:
\aaa{
\lb{perturba0}
\begin{pmatrix}
\frac{d \bar x}{dt}\\
\\
\frac{d \bar y}{dt}
\end{pmatrix}
= A \begin{pmatrix}
 \bar x\\
\\
\bar y
\end{pmatrix}
+ \begin{pmatrix}
-\bar x^2 -q\bar x\bar y\\
\\
\frac{\bar x\bar y}{k} + \bar y^2
\end{pmatrix}
+ \begin{pmatrix}
-(\bar x+a)\kappa(t)\\
\\
0
\end{pmatrix}
}
with
\aa{
A:=
\begin{pmatrix}
n+l_0 -2a -qb & -qa\\
\\
\frac{b}{k} & \frac{a}{k} + 2b - \frac{n-2k}{k}
\end{pmatrix}
}
The {stationary} points of $(LVS_{q,\nu_-})$ are $P_1(0,0), P_2(0,\frac{n-2k}{k}), P_3(n+l_0,0)$, and
\begin{equation}\lb{interiorcritipoint:minus}
P_4\left(\frac{q(n-2k)- k(n+l_0)}{q-k}, \frac{2k+l_0}{q-k}\right)=: P_4(\hat x,\hat  y).
\end{equation}
Note that, under the assumptions $2k<n,\, k<q$ and $-2k<l_0$, the first three critical points belong to $\Rnu^2_+:=\{(x,y)\in\Rnu^2:x\gt 0,\,y\gt 0\}$ and they are saddle points. The fourth critical point $(\hat x,\hat  y)$ belongs to the interior of $\Rnu^2_+$ if, and only if, $q> k(n+l_0)/(n-2k)$. Further, $(\hat x,\hat  y)$ is a stable node for $q> q^*(k,l_0)=\frac{(n+2)k+l_0(k+1)}{n-2k}$. It is not difficult to see that the (bounded) orbit $(x(t),y(t))$ of $(LVS_{q,\nu_-})$ starts at $P_3(n+l_0,0)$ (see \cite{SaVe17}).
\medbreak

\subsubsection{The flow lines of \rf{LVSrho}}
\lb{TSqnu}
We define
\begin{eqnarray*}
S(t,x,y)&:=&\nu(t)-x-q y,\;\; t\in\bar{\Rnu},\\
W(x,y)&:=&-\frac{n-2k}{k}+\frac{x}{k}+y
\end{eqnarray*}
and write \rf{LVSrho} in the form $x' = x S(t,x,y)$, $y' = y W(x,y)$.
For  a function $F:\Rnu^+\times\Rnu^+\to\Rnu$, we define
\[
F_0:=F^{-1}\{0\},\;\; F_+:=F^{-1}\{\Rnu^+\},\;\; F_-:=F^{-1}\{\Rnu^-\}.
\]
Observe that $W_0$ is the fixed straight line $y=\frac{n-2k}{k}-\frac{x}{k}$. Furthermore, we define
\[
G(x,y):=x+\frac{(n-2k)(q+1)}{k+1}\left(\frac{k}{n-2k}\,y-1\right).
\]
Note that $G_0$ is the straight line $y=-\frac{k+1}{k(q+1)}\,x+\frac{n-2k}{k}$ with intercepts $\left(\frac{(n-2k)(q+1)}{k+1},0\right)$ and $(0,\frac{n-2k}{k})$.

\begin{lemma}\lb{inward}
Assume $(\rho.2)$.
Let $\ffi$ be a solution of \rf{LVSrho} on $(T_0,T)$. 
\begin{itemize}
\item [i)] If $\ffi(t_0)\in G_{-}\cup G_0\;\;\mbox{for}\;\;t_0\gt -\infty$, then $\ffi(t)\in G_-$ for $t>t_0$.
\item [ii)] If $\ffi(t_0)\in G_{+}\cup G_0\;\;\mbox{for}\;\;t_0>-\infty,$ then $\ffi(t)\in G_+$ for $-\infty\lt t<t_0$.
\end{itemize}
\begin{proof} For the proof see Lemma 4.1 in \cite{MiSV19} (with the obvious changes).
\end{proof}
\end{lemma}

\subsubsection{Construction of a singular solution}
To construct a singular solution of \rf{Plambda}, we need the following technical lemma. 
For the proof, we redirect the reader to \cite[Lemma 4.3]{MiSV19},
where one needs to substitute $\mu$ with $l_0+2$ and $\frac{\mu e^{2t}}{1+e^{2t}}$
with $\kappa(t) = R(e^t) - l_0$.

\begin{lemma}\lb{S3L1}
Assume {conditions} ($\rho.1$)--($\rho.3$) hold.
Then, there exists a $t_0\in\Rnu$ such that 
\rf{LVSrho} admits a trajectory $(x(t),y(t))$  
with the property $x(t),y(t)\in \C^1(-\infty,t_0)$ and
\aa{
(x(t),y(t))\to(\hat x,\hat  y)\ \textrm{as}\ t\to -\infty,
}
where $(\hat x,\hat  y)$ is defined by \rf{interiorcritipoint:minus}.
\end{lemma}

{Now we are in position to construct a singular solution of \rf{Plambda}. For this, we use  the trajectory $(x(t),y(t))$ obtained in Lemma \ref{S3L1}. As we will see, the value $\td\la $ associated with the singular solution will depend on both the trajectory and the weight function $\rho$.}

\begin{proposition}
\lb{S3L2}
Assume that conditions $(\rho. 1)$--$(\rho. 3)$ hold.
Let $(x(t),y(t))$ be the trajectory obtained in Lemma \ref{S3L1}.
Then, for
\aaa{
\lb{tdla1111}
\la = \td \la:= \frac{c_{n,k}x(0)y(0)^k}{\rho(1)},
}
{Problem} \rf{Plambda}
possesses a singular solution $\td u$ with
the following asymptotic representation as $r\to 0:$
\aaa{
\lb{S3L2E12}
\td u(r)=-\left(\frac{c_{n,k}\hat x\hat  y^k}{\td\la \,K(0)}\right)^{\frac{1}{q-k}}r^{-\frac{2k+l_0}{q-k}}(1+o(1)).
}
\end{proposition}
\begin{proof}
Since $(x(t),y(t))\to(\hat x,\hat  y)$ as $t\to-\infty$ and $\rho(r)\sim K(0) r^{l_0}$ at $r=0$, formula \rf{inverse0} implies
\aaa{
\lb{S3L2E1+}
u(r)  :=1+w(r)
=1-\left(c_{n,k}^{-1}\la r^{2k}\rho(r)\right)^{-\frac{1}{q-k}}\left(x(t)y(t)^k\right)^{\frac{1}{q-k}}
\to-\infty \; \textrm{as}\ r\to 0.
}
Then, $u(r)$ is a singular solution to the equation in \rf{Plambda}.
Derivating the above identity in $r$, we obtain
\aa{
(c_{n,k}^{-1}\la)^{\frac{1}{q-k}} w'(r) &=\left(r^{2k}\rho(r)\right)^{-\frac{1}{q-k}}r^{-1}\frac{1}{q-k}(x(t)y(t)^k)^{\frac{1}{q-k}}
\left\{ 2k+\frac{r\rho'(r)}{\rho(r)}-\left(\frac{x'(t)}{x(t)}+k\frac{y'(t)}{y(t)}\right)\right\}\\
&=r^{-\frac{q+k}{q-k}}\rho(r)^{-\frac{1}{q-k}}\frac{1}{q-k}(x(t)y(t)^k)^{\frac{1}{q-k}}(q-k)y(t).
}
Therefore, since $\rho(r)\sim K(0) r^{l_0}$ at $r=0$,  
there exists a constant $C>0$ such that  as $r\to 0$,
\aaa{
\lb{S3L2E2}
|r^{n-k}(w'(r))^k|
\lt 
Cr^{\frac{q(n-2k)-k(n+l_0)}{q-k}}\to 0.
}
Since \rf{LVSrho} has {trajectories} on the $x$-axis and $y$-axis, 
by the uniqueness of a solution for \rf{LVSrho}, the {trajectory} $(x(t),y(t))$ is not on the $x$-axis nor on the $y$-axis. Hence $x(t)>0$ and $y(t)>0$ as long as the solution $(x(t),y(t))$ exists. By \rf{inverse0},
\aaa{
\lb{S3L2E3}
w(r)<0.
}
The equation in \rf{Plambda} can be rewritten as
\aa{
r^{n-k}(w'(r))^k-s^{n-k}(w'(s))^k=\int_s^r\tau^{n-1}c_{n,k}^{-1}\la\rho(\tau)(-w(\tau))^qd\tau.
}
By \rf{S3L2E2}, the integral on the right-hand side converges at $0$. Therefore,
\aaa{
\lb{S3L2E4}
r^{n-k}(w'(r))^k=\int_0^r\tau^{n-1}c_{n,k}^{-1}\la\rho(\tau)(-w(\tau))^qd\tau.
}
Next we define
\aa{
\bar r:=\sup\{\sigma>0:\ \textrm{the solution $w(r)$ of (\ref{Eq:IVP:0}) exists on $(0,\sg)$}\}.
}
We show by contradiction that $\bar r=\infty$.
Suppose $\bar r<\infty$.
By \rf{S3L2E3} and \rf{S3L2E4}, $w'(r)>0$ on $(0,\bar r)$. Take $r_0\in (0,\bar r)$.
Integrating equation \rf{Eq:IVP:0} from $r_0$ to $\bar r$, we obtain
\aa{
r^{n-k}(w'(r))^k = r_0^{n-k}(w'(r_0))^k+\int_{r_0}^r\tau^{n-1}c_{n,k}^{-1}\la\rho(\tau)(-w(\tau))^qd\tau >r_0^{n-k}(w'(r_0))^k.
}
Therefore,
$w'(r)>\left(\frac{r_0}{r}\right)^\frac{n-k}{k}w'(r_0)>0$.
Integrating the latter inequality gives
\aa{
0>w(r)>w(r_0)+\int_{r_0}^r\left(\frac{r_0}{\tau}\right)^\frac{n-k}{k}w'(r_0)\,d\tau>-\infty.
}
Hence, the finite limit $\lim_{r\uparrow\bar r}w(r)$ exists. 
By \rf{S3L2E4},  the finite limit $\lim_{r\uparrow\bar r}w'(r)$ also exists.
Since $w(\bar r)<0$, $w'(\bar r)\neq 0$, and $w$ satisfies \rf{Eq:IVP:0}, $w(r)$ can be locally defined as the solution of \rf{Eq:IVP:0} in a right neighborhood of $r=\bar r$. This contradicts the definition of $\bar r$, and therefore $\bar r=\infty$.

Since $w(r)$ is defined on $(0,\infty)$,
\rf{newtrans0} yields that $(x(t),y(t))$ is defined for all $t\in\Rnu$.
Evaluating \rf{S3L2E1+} as $r=1$, we obtain that $\td \la$ is indeed defined by \rf{tdla1111}.

Finally, we define $\td u$ by \rf{S3L2E1+} with $\la = \td \la$ obtaining
$(\td\la ,\td u(r))$ as a desired singular solution.
\end{proof}

\subsection{Intersection number and proof of Theorem \ref{S4L4}}
\lb{InterNumber}
Here we are interested in studying the intersection number between a regular and a singular solution of suitable equations. This intersection number is the most
important ingredient in the proof of
 Theorem \ref{S4L4} on the multiplicity of solutions to Problem \rf{Plambda}. 
We start by showing the existence of a regular solution to the 
initial value problem
\eq{
\lb{IVPw}
(r^{n-k}(w'(r))^k)'= \la \, c_{n,k}^{-1}r^{n-1}\rho(r)(-w(r))^q, & r>0,\\
w(0)=w_0\in (-\infty,0),\;\;
w'(0)=0
}
and its characterization via a proper trajectory of system \rf{LVSrho}.
The existence result for problem \rf{IVPw} will be used in the proof of Theorem \ref{S4L4}.
\begin{proposition}
\lb{prop35}
Assume that conditions {$(\rho. 1)$--$(\rho.3)$} hold. 
 Further assume that $\rho(r)$ has at most polynomial growth as $r\to\infty$.
Then there exists a unique global solution $w$ of \rf{IVPw} in the regularity class $C^2(0,\infty)\cap C^1[0,\infty)$. Furthermore, the function $w$ defined by \rf{inverse0} is the unique solution of \rf{IVPw} if, and only if, the {trajectory} $(x(t),y(t))$ of system \rf{LVSrho} given by \rf{newtrans0} starts at the point $P_3(n+l_0,0)$.
\end{proposition}
\begin{proof}
To apply Theorem 4.1 from \cite{ClMM98}, 
we define $B(r)=\int_0^r s^{\frac{(\al-\beta)(q+1)}{\beta+1}}(a(s)s^{\te})'ds$, $r>0$, with 
\aa{
\al=n-k;  \quad \beta=k, \quad  \gm=n-1, \quad 
a(s)=\la \, c_{n,k}^{-1}\rho(s), \quad  \te=\frac{\gm+1}{\beta+1}-\frac{(\al-\beta)(q+1)}{\beta+1}
} 
and note that 
the condition $B(r)\lt 0$ on $(0,\infty)$ follows from $(\rho.2)$. 
Indeed, we have
 $\te=n -\frac{(n-2k)(q+1)}{k+1}$, 
 $(a(s)s^\te)'= \la\, c_{n,k}^{-1} s^{\te-1}\rho(s)\big(\te+ R(s)\big)$, so for $s>0$,
\aa{
\te+R(s) <
n+l_0 -\frac{(n-2k)(q+1)}{k+1}<0,
}
where the last inequality follows from condition $(\rho.2)$.

Then the global existence of \rf{IVPw} follows from \cite[Theorem 4.1]{ClMM98}. 
The uniqueness follows from a {contraction mapping argument}, as in \cite{SaVe16}.

Next, let $w(r)$ be the unique solution of \rf{IVPw}. By \rf{newtrans0} and \rf{IVPw}, the function $y = y(t)$ satisfies
\aaa{
\lb{limi:y}
\lim_{t\to -\infty}y(t)=\lim_{r\to 0}r\,\frac{w'(r)}{-w(r)} = 0
}
and for $x=x(t)$, we have
$
\lim_{t\to -\infty}x(t)=\lim_{r\to 0}\la c_{n,k}^{-1}(-w(r))^q\frac{r^{k}\rho(r)}{(w'(r))^k}.
$
Now, by the equation in $\rf{IVPw}$ and L'H\^{o}\-pital's rule, we have
\eqm{
\lim_{r\to 0}\frac{r^{k}\rho(r)}{(w'(r))^k}&=\lim_{r\to 0}\frac{r^{n}\rho(r)}{r^{n-k}(w'(r))^k}=\lim_{r\to 0}\frac{r^n \rho'(r)+nr^{n-1}\rho(r)}{\la c_{n,k}^{-1} r^{n-1} \rho(r)(-w(r))^q}\\
&=\lim_{r\to 0}\frac{R(r)+n}{\la c_{n,k}^{-1} (-w(r))^q}=\lim_{r\to 0}\frac{\nu(\ln r)}{\la c_{n,k}^{-1} (-w(r))^q}=\frac{n+l_0}{\la c_{n,k}^{-1} (-w_0)^q}.
}
Thus, $\lim_{t\to -\infty}x(t)=n+l_0$. From here and \rf{limi:y}
we conclude that
\aaa{
\lb{start}
\lim_{t\to -\infty}(x(t),y(t))=(n+l_0,0)=P_3(n+l_0,0).
}
Conversely, suppose that $(x(t),y(t))\to P_3(n+l_0,0)$ as $t\to -\infty$. Rewriting the equation for $y$ in \rf{LVSrho}, we have
\aa{
y'=y\Big(-\frac{n-2k}{k}+\frac{x}{k}\Big)+y^2.
}
Let $z=y^{-1}$. Then, the above equation reduces to
\aa{
z'=z\Big(\frac{n-2k}{k}-\frac{x}{k}\Big)-1.
}
By simple computations, we obtain
\aa{
z(t)=z(t_0)e^{\gm(t_0)}-e^{\gm(t)}\int_{t_0}^{t}e^{-\gm(s)}ds
}
for some $t_0\in(T_0,T)$ and $\gm(t)=\int_{t_0}^{t}\left(\frac{n-2k}{k}-\frac{x}{k}\right)ds$.
Since $x(t)\to n+l_0$ as $t\to -\infty$, we see that 
$\gm(t)=-\frac{l_0+2k}{k} t +o(t)$, 
which also implies that $z(t)=O\left(e^{- \frac{l_0+2k}{k}\, t}\right)$ as $t\to -\infty$; in other words, we have $y(t)=O\left(e^{\frac{l_0+2k}{k}\, t}\right)$.
By the inverse transformation \rf{inverse0} and the fact that $\rho(r)\sim K(0) r^{l_0}$ at $r=0$, we obtain 
\begin{eqnarray*}
(-w(r))^{q-k}&=&\frac{c_{n,k}}{\la}\frac{x(\ln r)(y(\ln r))^k}{r^{2k}\rho(r)}=\frac{c_{n,k}}{\la}\frac{x(\ln r)\left(Cr^{\frac{l_0+2k}{k}}+o\left(r^{\frac{l_0+2k}{k}}\right)\right)^k}{r^{2k}\rho(r)}\\
&\to& \frac{c_{n,k}}{\la}(n+l_0)\td{C} = :(-w_0)^{q-k}\;\; (r\to 0),
\end{eqnarray*}
where $C$ and $\td{C}$ are positive constants independent of $t$.

On the other hand, differentiating the function in \rf{inverse0} with respect to $r$, we obtain
\aaa{
w'(r)&=\left(c_{n,k}^{-1}\la\,r^{2k}\rho(r)\right)^{-\frac{1}{q-k}}r^{-1}\frac{1}{q-k}(x(t)y(t)^k)^{\frac{1}{q-k}}
\left\{2k+\frac{r\rho'(r)}{\rho(r)}-\left(\frac{x'(t)}{x(t)}+k\frac{y'(t)}{y(t)}\right)\right\}\notag\\
&=-\frac{1}{q-k}\frac{w(r)}{r}\Big\{2k+R(r)-\Big(n+R(r)-x-qy+k\Big(-\frac{n-2k}{k}+\frac{x}{k}+y\Big)\Big)\Big\} \notag\\
&=-w(r)\,\frac{y(\ln r)}{r} \to 0 \;\; (r\to 0)
\lb{wr1111}
}
since $y(\ln r) = O(r^{\frac{l_0}{k} +2})$.
Hence, the function $w$, defined by \rf{inverse0}, is the unique solution of Problem \rf{IVPw} by the first statement of this proposition. This {ends} the proof.
\end{proof}
\begin{remark}
{
The condition on the polynomial growth of $\rho$ is required to satisfy
the condition $r^n|\rho'(r)| \lt C r^{{\mu}}$ (fulfilled for some constants $C>0$ and $\mu>-1$)
 {in order} to apply  Theorem 4.1 from \cite{ClMM98}. Remark that by $(\rho.2)$,
 $|\rho'(r)| \lt l_0 \rho(r) r^{-1}$ and, by $(\rho.1)$,  $\rho(r)$ is bounded on any 
 interval $[0,M]$, $M>0$. Therefore, the above condition, required by Theorem 4.1 
 from \cite{ClMM98}, is fulfilled if $\rho(r)$
 is of at most polynomial growth at $+\infty$.}
\end{remark}


Let $\td\la $ be as in Proposition \ref{S3L2} and consider the problem
\eq{
\lb{S4E1}
(r^{n-k}(U'(r))^k)'=c_{n,k}^{-1}\td\la  \, K(0) \, r^{n+l_0-1}(-U(r))^q, & r>0,\\
U(0)=-1,\quad 
U'(0)=0.
}
Further let 
\aa{
\td U (r):=-\left(\frac{c_{n,k}\hat x\hat  y^k}{\td\la \, K(0)}\right)^{\frac{1}{q-k}}
r^{-\frac{2k+l_0}{q-k}}.
}
Note that $\td U (r)$ is a singular solution of the first equation in \rf{S4E1}, {which is an Emden-Fowler-type equation corresponding to the system 
$(LVS_{q,\nu_-})$ with $\rho(r)= K(0)r^{l_0}$}.

In this case, substituting the constant trajectory $(x(t),y(t))=(\hat x,\hat y)$,
the number $\la=\td\la$, and  $\rho(r)= K(0)r^{l_0}$ into \rf{inverse0}, we obtain the function $\td U (r)$. 
Indeed, it suffices to compare the value $\td\la$ with the one defined in \cite[Theorem 3.1 (I)]{SaVe17}, as well as the corresponding singular solutions.

The next result say that there is infinitely many intersections between the singular solution 
$\td U (r)$ and the regular solution to \rf{S4E1}.

\begin{proposition}
\lb{S4P1}
Let $q^*(k,l_0)<q<q_{JL}(k,l_0)$ and let $U(r)$ be the unique 
solution to \rf{S4E1}.
Then
\aa{
\mc Z_{(0,\infty)}[\td U (\,\cdot\,)-U(\,\cdot\,)]=\infty,
}
where $\mc Z_{I}[\ffi]$ denotes the number of the zeros of the function $\ffi$ 
in the interval $I\sub \Rnu$. 
\end{proposition}

\begin{proof}
By the local analysis at the point $(\hat x,\hat y)$ discussed in {\cite[Section 6]{SaVe17}}, we see that this point is a stable spiral for $q^*(k,l_0)<q<q_{JL}(k,l_0)$. 
Therefore, there exists a strictly increasing sequence $\{t_n\}_{n=1}^{\infty}$ such that 
for all $n\in \Nnu$,
$y(t_n)=\hat  y$ and $x(t_2)<x(t_4)< \dots < x(t_{2n})< \dots  < 
\hat x <\dots < x(t_{2n-1})<\dots < x(t_3)<x(t_1)$. Let $r_n:=e^{t_n}$. By \rf{inverse0} with 
$\rho(r)= K(0) r^{l_0}$ {and $\la=\td\la$}, we have
\aa{
\frac{\td U (r_n)}{U(r_n)} 
= \left(\frac{\hat x}{x(t_n)}\right)^{\frac{1}{q-k}} 
\begin{cases}
<1, & \mbox{if $n$ is odd},\\
>1, & \mbox{if $n$ is even.}
\end{cases}
}
Therefore $\mc Z_{(0,\infty)}[\td U (\fdot)-U(\fdot)]=\infty$.
\end{proof}

\begin{lemma}\lb{S4L1}
Let $\td u(r)$ be the singular solution constructed in Proposition \ref{S3L2}. 
Further let $\td w(r)=\td u(r)-1$ and 
$(\mc F_a\td w)(r) =\frac{1}{a}\,\td w (\frac{r}{a^\gm})$ for $r>0$ and $a>0$,
where $\gm:=(q-k)/(2k+l_0)$.
Then, as $a\to\infty$,
\aaa{
\lb{S4L1E0}
(\mc F_{a}\td w )(r)\to\td U (r)\ \textrm{in}\ \C_{loc}(0,\infty).
}
\end{lemma}
\begin{proof}
By Proposition \ref{S3L2},
\[
\td w (r)=-\left(\frac{c_{n,k}\hat x\hat  y^k}{\td\la  \, K(0)}\right)^{\frac{1}{q-k}}r^{-\frac{1}{\gm}}(1+\epsilon(r)),
\]
where $\epsilon(r)$ satisfies $\limsup_{r\to 0}\epsilon(r)=0$.
Hence,
\aa{
\frac{1}{a}\,\td w \left(\frac{r}{a^{\gm}}\right)
=-\left(\frac{c_{n,k}\hat x\hat  y^k}{\td\la \, K(0)}\right)^{\frac{1}{q-k}}r^{-\frac{1}{\gm}}\left(1+\epsilon\left(\frac{r}{a^{\gm}}\right)\right)
\to \td U (r)\ \textrm{in}\ C_{loc}(0,\infty)\ \textrm{as}\ a\to\infty.
}
This concludes the proof.
\end{proof}

\begin{lemma}
\lb{S4L2}
Under  the assumptions of Proposition \ref{prop35}, we let
$w(r,a)$ be the unique solution to Problem \rf{IVPw} 
with $\la = \td \la$ and $w_0 = - a$
and let $(\mc F_a w)(r,a)$ be as in Lemma \ref{S4L1}.
Then, as $a\to\infty$,
\aaa{
\lb{S4L2E0}
(\mc F_a w)(r,a)\to U(r)\ \textrm{in} \ \C_{loc}[0,\infty),
}
where $U(r)$ is the unique solution to problem \rf{S4E1}.
\end{lemma}
\begin{remark}
In Proposition \ref{S4P1} and Lemma \ref{S4L2}, the existence
of unique solutions is known due to Proposition \ref{prop35}. 
\end{remark}
\begin{proof}[Proof of Lemma \ref{S4L2}]
Note that the function $\bar w(r,a):=(\mc F_a w)(r,a)$ is a solution to
\eqq{
(r^{n-k}(\bar w')^k)'=r^{n-1}c_{n,k}^{-1}\td\la \, r^{l_0}K(\frac{r}{a^\gm})(-\bar w)^q, & r>0,\\
\bar w(0,a)=-1,\quad
\bar w_r(0,a)=0.
}
Since $-a\lt w(r,a)\lt 0$ for $r\gt 0$, we see that $-1\lt\bar w(r,a)\lt 0$ for $r\gt 0$.
Integrating the above equation over $[0,r]$ and recalling that $0<K(r)< L$ for $r>0$ by
$(\rho.3)$, we obtain
\aaa{
\lb{S4L2E1}
r^{n-k}(\bar w_r)^k=\int_0^rc_{n,k}^{-1}\td\la \, s^{n+l_0-1}K
\left(\frac{s}{a^\gm}\right)(-\bar w)^qds.
}
Then,
\aa{
|\bar w_r(r,a)|
\lt\Big(r^{-n+k}\int_0^rc_{n,k}^{-1}\td\la \, L\, s^{n+l_0-1}ds\Big)^{1/k}
\lt\Big(\frac{c_{n,k}^{-1}\td\la \, L\,}{n+l_0}\Big)^{1/k}r^{\frac{k+l_0}{k}}.
}
Let $I\sub [0,\infty)$ be an arbitrary compact interval containing $0$.
Note that the family $\bar w(r,a)$ is uniformly bounded and equicontinuous on $I$.
By the Ascoli-Arzel\`a theorem, there exists a sequence $a_m\to +\infty$ (as $m\to\infty$) 
and a function $\bar w^*(r)\in C(I)$ such that as $m\to\infty$,
$\bar w(r,a_m)\to\bar w^*(r)\ \textrm{in}\ C(I)$.
By (\ref{S4L2E1}),
\aa{
\bar w(r,a_m)=-1+\int_0^r\Big( t^{-n+k}\int_0^tc_{n,k}^{-1}\td\la s^{n+l_0-1}K\Big(\frac{s}{a_m^\gm}\Big)(-\bar w(s,a_m))^qds\Big)^{1/k}dt.
}
Passing to the limit in this equation as $m\to\infty$ and noticing that
$K(\frac{s}{a_m^\gm})\to K(0)$, we obtain 
\aa{
\bar w^*(r)=-1+\int_0^r\Big(t^{-n+k}\int_0^tc_{n,k}^{-1}\td\la \,K(0)s^{n+l_0-1}(-\bar w^*(s))^qds\Big)^{1/k}dt\;\; \textrm{for}\ r\in I.
}
Therefore, $\bar w^*(r)\in \C^2(\mathring I)\cap \C^1(I)$ 
(here $\mathring I$ denotes  the interior of $I$)
and $\bar w^*(r)$ is the solution
to \rf{S4E1}.
By uniqueness, $\bar w^*(r)=U(r)$ for $r\gt 0$. Convergence \rf{S4L2E0}
holds by uniqueness of the limit point for the family $\mc F_a w$.
\end{proof}

\begin{lemma}\lb{S4L3}
Under the assumptions of Proposition \ref{S4P1},
\aaa{
\lb{S4L3E0}
\mc Z_{[0,1]}[\td w (\,\cdot\,)-w(\,\cdot\,,a)]\to\infty\ \textrm{as}\ a\to\infty.
}
\end{lemma}
\begin{proof}
Let $U$ and $\td U$ be the regular and singular solutions to the equation in \rf{S4E1} determined above.
Define $V =  U - \td U$. Then, $V$ satisfies the ODE
\aaa{
\lb{S4L3E2}
kr^{n-k}({\td U}')^{k-1} V'' + \big(k r^{n-k} U' \Phi_1
+(n-k) r^{n-k-1}\Phi_2\big)V' 
- r^{n+l_0-1}c_{n,k}^{-1}\td\la  K(0) \Phi_3 V = 0,
}
where $\Phi_1$, $\Phi_2$, and $\Phi_3$ are continuous {functions in $r$}.
Since $\td U'\neq 0$, the ODE \rf{S4L3E2} is of second order.
 We claim that 
 \aaa{
 \lb{prop1111}
 V'(r_0) \ne 0  \; \text{  whenever  } \; V(r_0) = 0.
 }
The proof of this claim is exposed in the beginning of the proof of Theorem
4.3 in  \cite{OSS22}. In this proof, one has to take  $f(x) = q\ln(-x)$,
$u = - U$ and $u^* = - \td U$.

By the property \rf{prop1111}, each zero of 
$V$ is simple, and furthermore,
the set of zeros of $V$ has no accumulation points. Therefore,
each zero is isolated.
By Proposition \ref{S4P1}, for every  $N\in \Nnu$, sufficiently large, 
there exists a number $M>0$ such that 
$\mc Z_{[0,M]}(\td U - U) \gt N+1$.
By Lemmas~\ref{S4L1} and \ref{S4L2}, as $a\to\infty$,
\aa{
\mc F_{a}\td w \to\td U  \quad \text{and} \quad \mc F_{a}w\to U
\ \ \textrm{in}\ \ \C_{loc}(0,\infty),
}
Therefore, in a neighborhood of each zero of $\td U -U$, there exists at least one zero of $\mc F_{a}\td w -\mc F_{a}w$ when $a$ is sufficiently large.
Hence,
\aa{
\mc Z_{[0,{M}]}\big(\mc F_a\td w -(\mc F_a w)(\,\cdot\,,a)\big)\gt N.
}
Since $\mc Z_{[0,M]}\big(\mc F_a\td w -(\mc F_{a}w)(\fdot,a)\big)=\mc Z_{[0,a^{-\gm}M]}\big(\td w -w(\fdot,a)\big)$,
where $\gm=(q-k)/(2k+l_0)$,
we obtain
\aa{
\mc Z_{[0,a^{-\gm}M]}[\td w -w(\fdot,a)]\gt N.
}
When $a>0$ is large, $[0,a^{-\gm}M]\sub [0,1]$, and hence,
\aa{
\mc Z_{[0,1]}[\td w -w(\fdot,a)]\gt\mc Z_{[0,a^{-\gm}M]}[\td w -w(\fdot,a)]\gt N.
}
The lemma is proved.
\end{proof}


\begin{proof}[Proof of Theorem \ref{S4L4}]
Let $\td \la$ be defined by \rf{tdla1111} and
let $w(r,a)$ be the solution to \rf{IVPw} with $\la=\td \la$
and $w_0 = -a$, whose existence and uniqueness
is known due to Proposition \ref{prop35}. 
It is straighforward to verify that $\bar w(r,a):=({\td\la }/{\la})^{1/(q-k)}w(r,a)$ 
is a solution to
\eqq{
(r^{n-k}(\bar w')^k)'=r^{n-1}c_{n,k}^{-1}\la\rho(r)(-\bar w)^q, \quad r>0,\\
\bar w(0,a)=-\left(\frac{\td\la }{\la}\right)^{1/(q-k)}\!a,
\qquad
\bar w_r(0,a)=0.
}
The reminder of the proof follows the lines of the proof of Theorem 2.2 in \cite{MiSV19}, 
so we omit it. 
\end{proof}

\section{$P_2$\,-, $P_3^+$-, and $P_4^+$-solutions to Problem \rf{P2P3}}\lb{tos}
\lb{secP2P3}
 In this section, we characterize {$P_2$\,-, $P_3^+$-, and $P_4^+$-solutions} to Problem \rf{P2P3}.
 Since we are only interested in the behavior of these solutions at a neighborhood 
 of $+\infty$, we will define these solutions on an interval $(M,+\infty)$
 for some $M>0$ large enough.
 \begin{definition}
A $\C^2(M,+\infty)$-solution $w(r)$ to Problem \rf{P2P3}
is called a $P_2$\,-, $P^+_3$-, or $P^+_4$-solution if the associated orbit
$(x(t),y(t))$, $t=\ln(r)$,  of the non-autonomous Lotka-\!Volterra system \rf{LVSrho} 
tends to $(0,\frac{n-2k}{k})$,  $(n+l_\infty,0)$, 
or $\big(\frac{q(n-2k)- k(n+l_\infty)}{q-k}, \frac{2k+l_\infty}{q-k}\big)$,
respectively, as $t\to+\infty$. 
\end{definition}
The main tool for {characterizing} these solutions is the associated 
non-autonomous Lotka-\!Volterra 
system \rf{LVSrho},  which it will be convenient to rewrite with the help
of the function $\zeta(t) = \nu(t) -  n-l_{\infty}$ in one of the following ways: 
\aaa{
\lb{lv1111}
\begin{cases}
\dot x = x(\nu_+ + \zeta(t) - x - qy),\\
\dot y = y(-\frac{n-2k}k +\frac{x}{k} + y);
\end{cases}
\qquad 
\begin{cases}
\dot x = x(\nu_+ + \zeta(t) - x - qy),\\
\dot y = y(-\dl +\frac{x-\nu_+}{k} + y).
\end{cases}
}
Note that $\lim_{t\to+\infty} \zeta(t) = 0$.

As we mentioned in {the Introduction}, an important role will played by the parameter
$\dl$, given by \rf{dl1111}. More specifically, we will show that $P_2$-solutions  
to \rf{P2P3} can exist
for all values of $\dl$; $P_3^+$-solutions can exist only if $\dl\gt 0$, and
$P_4^+$-solutions can exist only if $\dl < 0$.

\subsection{Classification of the stationary points of \rf{limit1111}}
\lb{lineariz1111}

In this subsection, we assume that $(\rho.4)$ is in force. Consider the autonomous system
 \eq{
 \lb{limit1111}
 \tag{$LVS_{q,\nu_{+}}$}
\dot x = x(n+l_\infty - x - qy),\\
\dot y = y(-\frac{n-2k}k +\frac{x}{k} + y)
}
whose right-hand side is the limit  of  \rf{LVSrho} as $t\to+\infty$. 
Let $P=(a,b)$ be a stationary point of \rf{limit1111}. Introduce the coordinates 
$\bar x=x-a$ and $\bar y =y-b$. Then, one can rewrite \rf{limit1111} as follows:
\begin{equation*}
\begin{pmatrix}
\frac{d \bar x}{dt}\\
\\
\frac{d \bar y}{dt}
\end{pmatrix}
= A \begin{pmatrix}
 \bar x\\
\\
\bar y
\end{pmatrix}
+ \begin{pmatrix}
-\bar x^2 -q\bar x\bar y\\
\\
\frac{\bar x\bar y}{k} + \bar y^2
\end{pmatrix}
\end{equation*}
with
\aa{
\lb{matrixA}
A:=
\begin{pmatrix}
n+l_\infty -2a -qb & -qa\\
\\
\frac{b}{k} & \frac{a}{k} + 2b - \frac{n-2k}{k}
\end{pmatrix}.
}
The critical points of $(LVS_{q,\nu_+})$ are $P_1(0,0), P_2(0,\frac{n-2k}{k}), P_3^+(n+l_\infty,0)$, and
\begin{equation}\lb{interiorcritipoint:plus}
P_4^+\left(\frac{q(n-2k)- k(n+l_\infty)}{q-k}, \frac{2k+l_\infty}{q-k}\right)
 = P_4^+(\td{x},\td{y}).
\end{equation}
Below we classify $P_1$, $P_2$, $P_3^+$ and $P_4^+$.
First, we note that the inequality $q\gt q^*(k,l_0)$ is equivalent to 
$\frac{n+l_0}{n-2k} \lt \frac{q+1}{k+1}$ and 
{$q > q^*(k,l_0)$  is equivalent to $\frac{n+l_0}{n-2k} < \frac{q+1}{k+1}$.}
By $(\rho.4)$, 
\aaa{
\lb{inq*1111}
\frac{n+l_\infty}{n-2k} < \frac{q+1}{k+1} < \frac{q}{k}.
}
\textit{Case $n+l_\infty>0$.}
It is straighforward to verify that $P_1$ and $P_2$ are saddle points. 
We will classify $P_3^+$ depending on the parameter $\dl$ defined by \rf{dl1111}.
If $\dl>0$, then $P_3^+$ is a node, if $\dl<0$, then $P_3^+$
is a saddle.  Furthermore, if $\dl = 0$, then $P_3^+ = P_4^+ = (n-2k,0)$.
In the latter case, the eigenvalues
of the matrix $A$ with $(a,b) = (n-2k,0)$ are $\la_1=0$, $\la_2 = -(n-2k)$.
According to the results of \cite{Cao} (Section 5, item (B-3) on p. 811),
the point $P_3^+ = P_4^+ = (n-2k,0)$ classifies as a saddle-node. 

Consider now $P_4^+(\td x,\td y)$, where $(\td x,\td y)$ is 
defined by \rf{interiorcritipoint:plus}.
 A {straightforward} computation shows that at $P_4^+$
 \aaa{
 \lb{a4444}
 A =\begin{pmatrix}
 -\td x   & -q\td x\\
\\
\frac{\td y}{k} & \td y
\end{pmatrix}.
 } 
If $\dl>0$ ($l_\infty<-2k$), then $\td x>0$ (by \rf{inq*1111}) and $\td y<0$, so it
 is {straightforward} to compute that
 for the eigenvalues of $A$ is holds that $\la_1>0$ and $\la_2<0$. Therefore,
 $P_4^+$ is a saddle.
 If $\dl = 0$, then $P_4^+=P_3^+$; this case was analyzed above. 
 Finally, if $\dl<0$, then $\td x>0$ and $\td y>0$. 
 Further, by \rf{inq*1111},
 $\td x - \td y > 0$. The characteristic equation takes the form
 $\la^2 + (\td x - \td y)\la + (\frac{q}{k} - 1) \td x\td y = 0$ with the discriminant
 \aa{
 D = (\td x - \td y)^2 - 4\Big(\frac{q}{k} -1\Big) \td x\td y = \td y^2\Big(\mu^2
  + 2\mu \Big(1-\frac{2q}{k}\Big) + 1\Big), \qquad \mu = \frac{\td x}{\td y}.
  }
Note that the  quadratic polynomial (with respect to $\mu$) on the right-hand side
can take negative, zero or positive values. Indeed, it has two positive roots. 
 Therefore,  $P_4^+$ can be either a focus or a node. 
 Let us find the condition on $l_\infty$ implying that $P_4^+$ is a node.
 {The roots of the polynomial are 
  $\mu_{1,2} = \frac{2q}{k} - 1\mp 2\sqrt{(\frac{q}{k})^2 - \frac{q}{k}}$ 
($\mu_1<\mu_2$).
Then, the condition of the positivity of $D$ is $\frac{\td x}{\td y}>\mu_2$ or
 $\frac{\td x}{\td y}<\mu_1$. Equivalently,
 it can be written as follows:
 \aa{
 l_\infty < \frac{q(n-2k) - k(n+2\mu_2)}{k+\mu_2}
 \quad \text{or} \quad 
 l_\infty > \frac{q(n-2k) - k(n+2\mu_1)}{k+\mu_1}.
 }
 Under this condition, we have that
$\la_1<0$, $\la_2<0$, and $\la_1\ne \la_2$, so
$P_4^+$ is a stable node.
Note that if $D=0$, then $P_4^+$ is a stable degenerate node. 
 }
\begin{remark}
The aforementioned condition on $l_\infty$ is equivalent to $q>q_{JL}(k,l_\infty)$. 
See Section 6 in \cite{SaVe17}, p. 702.
\end{remark}

\textit{Case $n+l_\infty<0$.} 
$P_1$ is now a node and $P_2$ remains a saddle. 
Note that if $n+l_\infty < 0$, then $\dl = -\frac{2k+l_\infty}{k}\gt \frac{n-2k}{k}>0$. 
Therefore, $P_3^+$ is a saddle. Furthermore, the analysis of the previous case
shows that $P_4^+$ is a saddle.

\textit{Case $n+l_\infty = 0$.} In this case,  $P_2$ and $P_4^+$ remain saddle
points. Note that $P_3^+ = P_1$. According to the classification in \cite{Cao},
Section 5, item (B-4-b-1-j) on p. 812, $P_3^+ = P_1$ is a saddle-node.

Remark that the {classification} of $P_1$, $P_3^+$ and $P_4^+$ depends on how
the  number $l_\infty$ is positioned with respect to $-n$ and $-2k$. 
On the other hand, we will only be interested in stationary points located
in the first quadrant since the orbits of \rf{LVSrho} are located in this quadrant
due to the transformation \rf{newtrans0}.
For the reader's convenience, we summarize the results of this subsection
in the following table:
\begin{table}[H]
\centering
\begin{tabularx}{\textwidth}{| X | X | X | X | X| X|} 
\hline
 \small $l_\infty < -n$  & \small $l_\infty = -n$ & 
 \small $-n<l_\infty < -2k$ & \small $l_\infty = -2k$ & \small $l_\infty > -2k$ \\
\hline
 \small  $P_2$ is a saddle, \phantom{sss}  $P_3^+$ is a saddle but $P_3^ +\notin\{x\gt 0\}$, 
  \phantom{sss} 
   $P_4^+$ is a saddle but 
$P_4^+\notin \{y\gt 0\}$,  \phantom{sss}
$P_1$ is a node
 & 
\small   $P_2$ is a saddle, \phantom{sss}
$P_4^+$ is a saddle but 
$P_4^+\notin \{y\gt 0\}$, \phantom{sss} 
$P_1 = P_3^+$ is a  \phantom{ssss} saddle-node
&
 \small  $P_1,P_2$ are saddles,
 $P_4^+$ is a saddle but 
$P_4^+\notin \{y\gt 0\}$, \phantom{sss} 
 $P_3^+$ is a node 
 &
\small   $P_1,P_2$ are saddles,
 $P_4^+ = P_3^+$ is a  \phantom{ssss} saddle-node
&
\small   $P_1, P_2,P_3^+$ are \phantom{sss} saddles, \phantom{sssssssss} 
$P_4^+$\hspace{-0.5mm} is either  a node or a fo\-cus
\\  \hline
\end{tabularx}
\caption{Classification of the stationary points of \rf{limit1111} depending on $l_\infty$.}
\lb{table:table1}
\end{table}

 \subsection{Useful results}
\lb{useful}
The following lemma is a slight reformulation of Lemma 5.5 from \cite{BattLi}.
\begin{lemma}
\lb{lem1111}
Consider the differential equation $\dot z = r(t) z + b(t)$ on $[t_0, +\infty)$. Assume
$lim_{t\to+\infty} r(t) = r_0<0$ and $\lim_{t\to+\infty} b(t) = b_0$, $b_0\in\Rnu$. Then,
\aa{
\lim_{t\to+\infty} z(t) = -\frac{b_0}{r_0}.
}
\end{lemma}
\begin{proof}
The result follows immediately from Lemma 5.5 if we rewrite the ODE with respect to $\td z(t) = z(-t)$
and introduce the new coefficients $\td r(t) = -r(-t)$ and $\td b(t) = -b(-t)$.
\end{proof}
  \begin{proposition}
  \lb{prop-2k1111}
 Suppose  $l_\infty> - 2k$ (equivalently, $\dl<0$).
Then a solution to Problem \rf{P2P3} of class $\C^2(M,+\infty)$ 
cannot be a $P_3^+$-solution. 
\end{proposition}
\begin{proof}
Suppose 
$(x(t),y(t))$ is the orbit  of \rf{LVSrho} associated to a $P_3^+$-solution.
Computing the derivative of $\frac1y$ and using \rf{LVSrho}, we obtain
\aaa{
\lb{1111y}
\Big(\frac1y\Big)'  = -\frac1y \frac{\dot y}{y} 
= \frac1y\Big(\dl +\frac{\nu_+-x}k\Big) - 1.
}
Since $x\to \nu_+$ as $t\to+\infty$,  by Lemma \ref{lem1111}, 
$\lim_{t\to+\infty} \frac1{y(t)} = \frac1{\dl}<0$.
This contradicts to the fact that $\lim_{t\to+\infty} y(t) = 0$.
\end{proof}
\begin{lemma}
\lb{lemp2222}
Let $(\rho. 1)$, $(\rho. 2)$, $(\rho. 4)$, $(\rho. 5)$ hold.
Further let $(x,y)$ be  either an orbit of \rf{LVSrho} or an orbit of \rf{limit1111} 
tending to $P_2$ as $t\to +\infty$ in such a way that $x(t)>0$
 in a neighborhood of $+\infty$.  Then, 
in a neighborhood of $P_2$, $y$ can be represented as a function of $x$,
which we denote by $\hat y(x)$. If, moreover, $\hat y'(0 )$ exists,
then, in a neighborhood of $P_2$, the graph of $\hat y(x)$ lies in $G_+$ and
\aa{
\hat y'(0) = -\frac{n-2k}{k^2\gm + k(n-2k)}.
}
\end{lemma}
\begin{proof}
Note that  by \rf{inq*1111},
$\lim_{t\to+\infty} (\nu_+ + \zeta(t) - x - qy) = n+l_\infty - \frac{q}{k}(n-2k)<0$.	
Therefore, in a neighborhood of $+\infty$, $\dot x>0$ 
and one can express $t$ as a function of $x$.   In this 
neighborhood of $+\infty$,  we introduce the function
\aa{
\hat y(x) = y(t(x)),
}
where $t(x)$ is the inverse function to $x(t)$. Next, since $\hat y'(0)$ exists,
from \rf{LVSrho} we obtain
\aa{
\hat y'(0) = \lim_{x\to 0+} \hat y'(x)  = \lim_{t\to+\infty}\frac{\dot y(t)}{\dot x(t)} = \frac{\frac{n-2k}{k}\big(\frac1{k} +
 \lim_{x\to 0+}\frac{\hat y - \frac{n-2k}{k}}{x}\big)}{n+l_\infty - \frac{q}{k}(n-2k)} 
 = - \frac{(n-2k)(1+k \hat y'(0))}{k^2\gm}, 
}
where $\gm = \frac{q}{k}(n-2k) - (n+l_\infty)>0$ (by \rf{inq*1111}).
From the above equation, we obtain
the expression for $\hat y'(0)$.
Using this expression, we compare $\hat y'(0)$ with the slope of $G_0$ which equals
$-\frac{k+1}{k(q+1)}$.  Observe that $\frac1{|\hat y'(0)|} > \frac{k(q+1)}{k+1}$.
Indeed, the latter inequality is {equivalent} to $\frac{n+l_\infty}{n-2k} < \frac{q+1}{k+1}$,
which is the same as \rf{inq*1111}. This implies that in a neighborhood of $P_2$,
the graph of $\hat y(x)$ lies in $G_+$.
\end{proof}

The following theorem by Thieme will be useful.
Although the theorem holds for a large class
of {asymptotically} autonomous 2-dimensional systems (see Theorem 1.5 in \cite{Thie94}), 
we will formulate it only for systems \rf{LVSrho} and \rf{limit1111}. We take into account
that we have at most four equilibrium points, so the assumptions of Thieme's theorem
are fulfilled.
\begin{theorem}
\lb{tie1111}
For the  $\om$-limit set $\om$ of \rf{LVSrho}, the following trichotomy holds:
\bi
\item[\rm (a)] $\om$ consists of an equilibrium point of \rf{limit1111}.
\item[\rm (b)] $\om$ is the union of periodic orbits of \rf{limit1111} and possibly centers 
of \rf{limit1111} surrounded by periodic orbits lying in $\om$.
\item[\rm (c)] $\om$ contains equilibria of \rf{limit1111} that are {cyclically} chained
to each other by orbits of \rf{limit1111}.
\ei
\end{theorem}
The following corollary of Theorem \ref{tie1111} will be useful.
\begin{corollary}
\lb{pro1111}
Assume $(\rho.1)$, $(\rho.2)$, $(\rho.4)$, $(\rho.5)$. Then,
for orbits of \rf{LVSrho} associated to solutions of \rf{P2P3}, only 
option (a) of Theorem \ref{tie1111} can be realized. Furthermore, if $l_\infty{\lt} -2k$, 
the $\om$-limit set of \rf{LVSrho} is either 
$P_2$ or $P_3^+$; if $l_\infty>-2k$, 
the $\om$-limit set of \rf{LVSrho} is either $P_2$ or $P_4^+$.
\end{corollary}
\begin{proof}
Note that any solution to Problem \rf{P2P3}
 gives rise to the associated orbit of \rf{LVSrho}
totally located in the first quadrant. 
Therefore, the $\om$-limit set of such an orbit is contained in 
the domain $\{x\gt 0\} \cap \{y\gt 0\}$. We now analyze 
the $\om$-limit set of \rf{LVSrho} by using Theorem \ref{tie1111}.

First, we note that according to Theorem 2.1 in \cite{Cao}, \rf{LVSrho} does not
have non-trivial closed orbits in the first quadrant. This excludes option (b) 
of Theorem \ref{tie1111}. Theorem 2.1 also implies that in the hypothesis that
option (c) is realized, the $\om$-limit set of  \rf{LVSrho} cannot have homoclinic
orbits connecting saddles of  \rf{LVSrho} to themselves.

First, consider the case $l_\infty\lt -2k$  or, equivalently, $\dl\gt 0$. 
Let us show that this condition
fully excludes option (c). For this, we analyze the behavior of heteroclinic orbits 
of \rf{limit1111}  that can possibly start or end at the saddle points of \rf{limit1111}.
Since $l_\infty\lt -2k$, we have to separately analyze the situation
described in each of the first four columns of Table \ref{table:table1}. Note that
we can immediately disregard the first two columns since there is
only one saddle point $P_2$ in the domain $\{x\gt 0\} \cap \{y\gt 0\}$.
Thus, without loss of generality, we assume that $n+l_\infty>0$.
Note that we have to consider only the saddle points $P_1$ and $P_2$.
Indeed, if $l_\infty<-2k$, $P_4^+$ is also a saddle, but
$P_4^+ \notin \{y\gt 0\}$; so orbits starting or ending at $P_4^+$ are excluded.
{If $l_\infty = -2k$, then $P_3^+ = P_4^+$ is a saddle-node.} 
Thus, we analyze heteroclinic orbits that can possibly connect $P_1$ and $P_2$.

Below, $(x(t),y(t))$ denotes such an orbit.
Suppose first that  $x(t)$ and $y(t)$ are not identically equal zero 
at a neighborhood of $P_1$.
Then, at $P_1$ we have $(\ln x)' = n+l_\infty>0$ 
and $(\ln y)' = - \frac{n-2k}{k}<0$. 
This implies that as $t\to +\infty$,
$x$ increases and $y$ decreases. Therefore, any 
 heteroclinic orbit can enter $P_1$ as $t\to+\infty$ only from 
 quadrant ii. In a similar manner,  any 
 heteroclinic orbit can enter $P_1$ as $t\to-\infty$
 only from quadrant iv. Thus, the aforementioned case is excluded since
 the respective orbits cannot attract trajectories of \rf{LVSrho} corresponding
 to solutions of \rf{P2P3}.
 Next, if $x(t)$ equals zero in a neighborhood of $P_1$,
 then $y(t)$,  in a neighborhood of $P_1$, solves the ODE $\dot y = y^2 - \frac{n-2k}{k} y$ whose solution is
 \aaa{
 \lb{y2222}
 y(t) = \frac{n-2k}{k}\Big(1+ C e^{\frac{n-2k}k t}\Big)^{-1},
 }
 where $C\ne 0$ is a constant. If $C=0$, then $(x(t),y(t)) = P_2$
 is a trivial orbit and shall be considered within the analysis of case (a) of
 Theorem \ref{tie1111}. Note that, 
 by uniqueness, $(0, y(t))$, where $y(t)$ is given by \rf{y2222},
 is the only solution  to \rf{limit1111} with the property that
 $x=0$ in a neighborhood of $P_1$. It represents a heteroclinic orbit
that starts at $P_2$ at $t=-\infty$ and ends at $P_1$ at $t=+\infty$.
 Furthermore, if $y(t)=0$ in a neighborhood of $P_1$, then $x(t)$, 
 in a neighborhood of $P_1$, solves the 
 ODE $\dot x = x(n+l_\infty - x)$ whose solution is
 \aaa{
 \lb{x2222}
 x(t) = \frac{n+l_\infty}{C e^{-(n+l_\infty)t} + 1}, 
 }
 where $C\ne 0$ is a constant. Remark that 
 if $C=0$, then $(x(t),y(t)) = P_3^+$
 is a trivial orbit and shall be considered within the analysis of case (a) of
 Theorem \ref{tie1111}.  Also remark that 
 by uniqueness, $(x(t), 0)$, where $x(t)$ is given by \rf{x2222},
 is the only solution  to \rf{limit1111} with the property that
 $y=0$ in a neighborhood of $P_1$. However, the above-described orbit connects
 $P_1$ with $P_3^+$, not with $P_2$.
 
 Therefore, the segment $[0,\frac{n-2k}k]$ is the only heteroclinic orbit connecting
 $P_1$ and $P_2$.
 Let us prove that this orbit cannot belong to the $\om$-limit set 
 of any trajectory of \rf{LVSrho}. Arguing by contradiction, we assume
  that $(\td x(t),\td y(t))$ is such a trajectory.
   Fix $\eps\in (0, n-2k)$, sufficiently small, and find $\tau_1>0$ such that $\td x(t)<\eps$
  for $t>\tau_1$. Fix $y_0\in (0,\frac{n-2k-\eps}{k})$. There exists a sequence
  $t_N\to +\infty$ such that $\td y(t_N)\to y_0$.  Find $\tau_2>\tau_1$ such that
  $\td y(\tau_2)< \frac{n-2k-\eps}{k}$. It holds that ${\td y'}(\tau_2)<0$, 
  and therefore, $\td y' < 0$ on $(\tau_2,+\infty)$. Since $(0,0)$
  cannot be a limit point of any trajectory of \rf{lv1111},
  $y_1 = \lim_{t\to+\infty} \td y(t) >0$. Hence, the interval
  $[0,y_1)$ cannot belong to the $\om$-limit set of $(\td x(t),\td y(t))$
  which is a contradiction.

Thus, we proved that option {(c)} of Theorem \ref{tie1111} is also excluded.
  The only remaining option is (a). Now we have to analyze trajectories of
  \rf{LVSrho} possibly ending at the stationary points $P_i$, $i=1,2,3,4$.
  The same argument as we used in the analysis of heteroclinic 
  orbits of \rf{limit1111} shows that orbits of \rf{LVSrho} cannot end at $P_1$ or $P_4^+$,
  where $P_4^+$ is considered only when $l_\infty<-2k$. Thus,
  the $\om$-limit set of \rf{LVSrho} can only consist of $P_2$ or $P_3^+$.
  
  {Suppose now $l_\infty>-2k$ or, equivalently, $\dl<0$. Suppose item (c)
  of Theorem  \ref{tie1111}  takes place. This situation
  is reflected in the fifth column of Table \ref{table:table1}.
   The $\om$-limit set consists then of three saddle {points} $P_1$, $P_2$, and $P_3^+$
   and heteroclinic orbits connecting them. In particular, by the above argument,
   the orbit $(0,y(t))$, where $y(t)$ is given by \rf{y2222}, 
   goes from $P_2$ to $P_1$; and the orbit $(x(t),0)$, 
   where $x(t)$ is given by \rf{x2222}, goes from $P_2$ to $P_3^+$. 
   Since $[P_2,P_1]$ and $[P_1,P_3^+]$ attract trajectories $\ffi(t)$ of \rf{LVSrho},
   by Lemma \ref{inward},  it holds
   that $\ffi(t) \in G_-\cap \{x\gt 0\} \cap \{y\gt 0\}$ for all $t>t_0$ for some $t_0\in\Rnu$. 
   Therefore,
   any heteroclinic orbit, name it $\psi(t)$, starting at $P_3^+$ and ending at $P_2$, 
   should be located
   in the domain $(G_-\cup G_0) \cap \{x\gt 0\} \cap \{y\gt 0\}$. 
   However, Lemma \ref{lemp2222} implies that in a neighborhood of 
   $+\infty$, $\psi(t)$ enters $P_2$
   from $G_+$. Therefore, $\psi(t)$ cannot belong to the $\om$-limit set of \rf{LVSrho}.
   The conditions of Lemma \ref{lemp2222} are fulfilled since
   in a neighborhood of $P_2$, $\psi(t)$ is represented
   as a graph of function $\hat y(x)$ and $\hat y'(0)$ exists by
   item {\it (ii)} of Theorem \ref{p2222}, applied to system \rf{limit1111}. Indeed, 
   the asymptotic  representations obtained in item {\it (ii)}  of Theorem \ref{p2222}
   remain valid for $\psi(t)$, since the proof also works in the situation when $\zeta(t) = 0$.}
  Thus, we obtained a contradiction to option (c).
   The only remaining option is (a). Since the coordinates of $P_4^+$
   are positive and $P_4^+\in G_-$,
   it can belong to the $\om$-limit set of \rf{LVSrho}. 
   To see that $P_4^+\in G_-$, we substitute 
   $y=\frac{2k+l_\infty}{q-k}$ (the ordinate of $P_4^+$) into the right-hand side of the 
   equation for $G_0$: 
   $x = - \frac{k(q+1)}{k+1} y + \frac{q+1}{k+1}(n-2k)$.  By \rf{inq*1111}, we obtain
   \aa{
   - \frac{k(q+1)}{k+1}\frac{2k+l_\infty}{q-k} + \frac{q+1}{k+1}(n-2k)
   =\frac{q+1}{k+1} \, \frac{q(n-2k) - k(n+l_\infty)}{q-k} > \frac{q(n-2k) - k(n+l_\infty)}{q-k}.
   }
   This implies that $P_4^+\in G_-$. Further, by Lemma \ref{prop-2k1111},
   $P_3^+$ cannot belong to the $\om$-limit set. Therefore, if $l_\infty >-2k$, 
   the $\om$-limit  set of \rf{LVSrho}  consists either of $P_2$ or $P_4^+$.
  \end{proof}
  \begin{remark}
  \lb{rm1111}
  Remark that, under
  $(\rho.1)$, $(\rho.2)$, $(\rho.4)$, $(\rho.5)$,
   by the existence result obtained in {Proposition \ref{prop35}}
  global solutions to Problem \rf{IVPw} exist. Furthermore, by Corollary \ref{pro1111},
  these solutions classify as  $P_2$, $P_3^+$, or $P_4^+$. 
    \end{remark}
 
 \begin{lemma}
\lb{krho}
Assume $(\rho. 5)$. Then,
there exists a constant $c_\rho>0$ such that as $r\to+\infty$,
\aa{
\rho(r) = c_\rho r^{l_\infty}(1+o(1)).
}
\end{lemma}
\begin{proof}
By the definition of $R(r)$,
\aa{
\ln \rho(r) = \ln \rho_0 + \int_{r_0}^r \frac{R(s)-l_\infty}{s} ds 
+ l_\infty\int_{r_0}^r \frac{ds}{s}.
}
By $(\rho. 5)$, the integral containing $R(s)$ converges as $r\to+\infty$,
which implies the statement.
\end{proof}
In what follows, it will be rather convenient to formulate a version of
 assumption $(\rho. 6)$ in terms of the function $\zeta(t)=R(e^t)-l_\infty$.
\bi
\item[$(\rho. 6)$] One of the assumptions, (1) or (2), is fulfilled:
\bi
\item[(1)]\; $\lim_{t\to+\infty} e^{\dl t} \zeta(t) \in [0,+\infty)$ and  
$\nu_+> \dl$;
\item[(2)]\;  $\lim_{t\to+\infty} e^{\dl t} \zeta(t) = +\infty$,
the limit
$\hat \nu = \lim\limits_{t\to+\infty}\big(\!\!-\!\frac{\zeta'(t)}{\zeta(t)}\big)$ exists, 
and  $\nu_+> \hat \nu$.
\ei
\ei
First of all note that by L'Hopital's rule and since $r=e^t$,
\aaa{
\lb{lms}
\lim_{r\to+\infty} \frac{\ln |R(r) - l_\infty|}{\ln r} = \lim_{t\to+\infty}
\frac{\ln |\zeta(t)|}{t}
= \lim_{t\to+\infty} \frac{\zeta'(t)}{\zeta(t)}.
}
For the purpose of examples of weights given in Subsection \ref{examples2222}, 
we formulate an alternative assumption  
$(\rho. 6')$ which is stronger than $(\rho. 6)$, but it is easier to be verified.
\bi
\item[$(\rho. 6')$] Define $\psi(r) = r^\teta |R(r) - l_\infty|$. We assume 
(a) $\liminf_{r\to+\infty} \psi(r) >0$;
(b) $\nu_+>\min\{\dl,\teta\}$; and (c) $\lim_{r\to+\infty} \psi(r)>0$
exists in the case $\dl =\teta$.
\ei
\begin{lemma}
\lb{rem1111}
$(\rho. 6')$ implies $(\rho. 6)$.
\end{lemma}
\begin{proof}
Note that options (1) and (2) in the assumption $(\rho. 6)$ correspond to the cases
$\dl\lt \teta$ and $\dl>\teta$, respectively. 
Further we note that for all values of $\dl$ and $\teta$,
the limit
$\lim_{r\to+\infty} r^\dl |R(r) - l_\infty| = \ell$ exists, finite or infinite. 
More specifically, if $\dl< \teta$, then $\ell = 0$; if $\dl = \teta$, then $\ell \in (0,+\infty)$;
if $\dl>\teta$, then $\ell = +\infty$.
In the latter case, by the definition of $\psi(r)$, L'Hopital's rule,  and \rf{lms},
\aa{
\teta = -\frac{\ln r^{-\teta}}{\ln r} - \lim_{r\to+\infty} \frac{\ln \psi(r)}{\ln r}  
= -\lim_{r\to+\infty} \frac{\ln |R(r) - l_\infty|}{\ln r} 
= -\lim_{t\to+\infty} \frac{\zeta'(t)}{\zeta(t)} = \hat \nu.
}
This proves $(\rho. 6)$.
\end{proof}
\begin{corollary}
\lb{cor1111}
Let $\liminf_{r\to+\infty} \psi(r) >0 $ and $\dl\ne \teta$.
Then $(\rho. 6)$ is {equivalent} to
\bi
\item[$(\rho. 6')$] $\nu_+>\min\{\dl,\teta\}$.
\ei
\end{corollary}
\begin{corollary}
\lb{cor2222}
Let $\lim_{r\to+\infty} \psi(r)>0$ exist and $\dl = \teta$. Then, $(\rho. 6)$ is {equivalent} to
$(\rho. 6')$, where the latter is reduced to the inequality $\dl < \frac{n-2k}{k+1}$.
\end{corollary}

\subsection{$P_2$\,-solutions}
From Remark \ref{rm1111}, we know that if $q\gt q^*(k,l_0)$, then 
solutions to Problem \rf{P2P3} exist 
and can classify as $P_2$, $P_3^+$, or $P_4^+$.
In this subsection, we characterize $P_2$\,-solutions. 

\subsubsection{Proof of Theorem \ref{p2222}}
We are ready to prove Theorem \ref{p2222}.
The proof uses some ideas from the related result in \cite{BattLi} (Theorem 7.2).
\begin{remark}
Note that the number $\gm =  \frac{q}{k}(n-2k) - (n+l_\infty)$, defined in the statement of
Theorem \ref{p2222}, is positive by \rf{inq*1111} and $(\rho.2)$.
\end{remark}
\begin{proof}[Proof of Theorem \ref{p2222}]
\textit{Step 1. (i) $\to$ (ii)}. Suppose {that} $w(r)$ is a $P_2$\,-solution, i.e.,
as $t\to+\infty$, $x(t)\to 0$, $y(t)\to \frac{n-2k}k$. Let us obtain the {asymptotic}
representations \rf{p2xy}.
In a  {neighborhood} of $P_2(0,\frac{n-2k}k)$, one has the representation 
of  \rf{LVSrho} as a time-dependent perturbation of  \rf{limit1111}:
\begin{equation}
\lb{perturba}
\begin{pmatrix}
\frac{d \bar x}{dt}\\
\\
\frac{d \bar y}{dt}
\end{pmatrix}
= A \begin{pmatrix}
 \bar x\\
\\
\bar y
\end{pmatrix}
+ \begin{pmatrix}
-\bar x^2 -q\bar x\bar y\\
\\
\frac{\bar x\bar y}{k} + \bar y^2
\end{pmatrix}
+ \begin{pmatrix}
\bar x\zeta(t)\\
\\
0
\end{pmatrix}
\end{equation}
with $\bar x = x$, $\bar y = y-\frac{n-2k}{k}$, and the matrix 
\aa{
A=
\begin{pmatrix}
-\gm & 0\\
\\
\frac{n-2k}{k^2} &  \frac{n-2k}{k}
\end{pmatrix}.
}
Rewrite \rf{perturba} with respect to $\psi(t) = (x(t),\bar y(t))$ as follows:
\aaa{
\lb{Apsi1111}
\dot \psi(t) = A\psi(t) + \dlt(\psi(t)) + \eta(t,\psi(t)),
}
where $\dlt(\psi(t))$ and $\eta(t,\psi(t))$ denote the two last terms 
on the right-hand side of \rf{perturba}, respectively.
Note that $\la_1 = -\gm <0$ and $\la_2 = \frac{n-2k}{k}>0$
are the eigenvalues of $A$. Since $\la_1$ and $\la_2$ are real and different,
the corresponding unit eigenvectors  $v_1$ and $v_2$ are non-collinear.
Therefore, each vector $z\in\Rnu^2$ is a linear combination $z= z_1 v_1 + z_2 v_2$.
Define the projection operators $P_1z = z_1v_1$ and $P_2 z = z_2v_2$.
Since $P_1$ and $P_2$ are linear and continuous, there exists a constant $C>0$
such that
\aaa{
\lb{proj1111}
 |P_i z| \lt C |z|, \quad  i=1,2.
}
 Rewritng \rf{Apsi1111} in the semigroup form and taking into account that
$P_i e^{tA} = e^{tA} P_i = e^{t\la_i} P_i$, $i=1,2$,  we obtain
\mmm{
\lb{psi2222}
\psi(t) = e^{\la_1(t-\tau)}P_1\psi(\tau) + e^{\la_2(t-{\bar \tau})}P_2\psi({\bar \tau})
+ \int_{\tau}^t e^{\la_1(t-s)}P_1 \dlt(\psi(s)) ds\\
+  \int_{{\bar \tau}}^t e^{\la_2(t-s)}P_2 \dlt(\psi(s)) ds
+ \int_{\tau}^t e^{\la_1(t-s)}P_1 \eta(s,\psi(s)) ds
+  \int_{{\bar \tau}}^t e^{\la_2(t-s)}P_2 \eta(s,\psi(s)) ds,
}
where $\tau$ and $\bar\tau$ are arbitrary positive numbers. Above,
we added the equations for $P_1\psi(t)$ and $P_2\psi(t)$.
Note that as $t\to+\infty$, $\psi(t)\to 0$, $\dlt(\psi(t))\to 0$, and $\eta(t,\psi(t))\to 0$.
Fix $\tau>0$ such that for all $t>\tau$, $\psi(t)$, $\dlt(\psi(t))$, and  $\eta(t,\psi(t))$
remain bounded. For $t>\tau$, in \rf{psi2222}, we pass to the limit as ${\bar \tau}\to+\infty$. 
We write the resulting equation after multiplying the both parts by $e^{-\la_1 t}$:
\mmm{
\lb{psi1111}
e^{-\la_1 t}\psi(t) = e^{-\la_1\tau}P_1\psi(\tau) 
+ \int_{\tau}^t e^{-\la_1s}P_1 \big(\dlt(\psi(s))+  \eta(s,\psi(s))\big) ds\\
-  e^{-\la_1 t}\int_t^{+\infty} e^{\la_2(t-s)}P_2 \big(\dlt(\psi(s)) + \eta(s,\psi(s))\big)ds
= e^{-\la_1\tau}P_1\psi(\tau)  + I(t) + J(t).
}
Let us show that $I(t)$ and $J(t)$ converge to finite limits.  To this end, we find an exponentially decaying 
estimate for $|\psi(t)|$ by using Lemma 7.1 from \cite{BattLi}. First, we find $\sg>0$
such that 
\aaa{
\lb{beta1111}
\beta: = C\sg\Big(\frac1{|\la_1|} + \frac1{\la_2}\Big)<1.
}
Note that 
\aaa{
\lb{est111}
|\dlt(\psi(t))| \lt (|x(t)| + q|\bar y(t)|)|\psi(t)| 
\quad \text{and} \quad 
|\eta(t,\psi(t))| \lt |\zeta(t)| |\psi(t)|;
}
so the number $\tau$ in \rf{psi1111} can be chosen in such a way that 
 $|\dlt(\psi(t))| + |\eta(t,\psi(t))| < \sg |\psi(t)|$ for all $t\gt \tau$. Therefore,
 by \rf{proj1111}, 
 \aaa{
 \lb{estimate1111}
 P_i \big(\dlt(\psi(s))+  \eta(s,\psi(s))\big) \lt C\sg |\psi(t)|, \qquad i=1,2.
 }
Estimate \rf{estimate1111} and equation \rf{psi1111} imply that there exist  a constant $K_1>0$
such that
\aa{
|\psi(t)| \lt K_1 e^{\la_1 t} + C\sg \int_\tau^t e^{\la_1 (t-s)}|\psi(s)| ds + C\sg
 \int_0^{+\infty} e^{-\la_2 s}|\psi(s+t)| ds.
}
By Lemma 7.1 from \cite{BattLi}, there exist a constant $K_2>0$ such that
\aaa{
\lb{exp1111}
|\psi(t)| \lt K_2 e^{(\frac{C\sg}{1-\beta} + \la_1)t}.
}
Furthermore,  \rf{est111} imply that there exists a constant $K_3>0$ such that
\aaa{
\lb{est2222}
|\dlt(\psi(t))| \lt (1+q) |\psi(t)|^2 
\quad \text{and} \quad 
|\eta(t,\psi(t))| \lt K_3e^{-\teta t} |\psi(t)|.
}
By \rf{exp1111} and \rf{est2222}, we have the following estimates for $I(t)$ and $J(t)$ in \rf{psi1111}:
\aa{
& |I(t)| \lt K_4 \int_\tau^t \big(e^{(\frac{2C\sg}{1-\beta} + \la_1)s} + e^{(\frac{C\sg}{1-\beta} - \teta)s}\big) ds, \\
& |J(t)| \lt K_5 \, e^{(\la_2-\la_1)t}\int_t^{+\infty} 
\big(e^{(\frac{2C\sg}{1-\beta} + 2\la_1 - \la_2 )s} + e^{(\frac{C\sg}{1-\beta} - \teta +\la_1 - \la_2)s}\big)ds,
}
where $K_4,K_5>0$ are constants.
In \rf{beta1111}, we can take $\sg>0$ even smaller so that $\frac{2C\sg}{1-\beta} + \la_1<0$
and $\frac{C\sg}{1-\beta} - \teta<0$. With this choice of $\sg$, $\lim_{t\to+\infty} I(t)$ is finite
and $\lim_{t\to+\infty} J(t) = 0$. Therefore, $\lim_{t\to+\infty} e^{\gm t} \psi(t)$ exists.
We denote it by $\vec{c_1}{c_2}$.  Note that $c_1 = \lim_{t\to+\infty} e^{-\gm t} x(t) \gt 0$.
In {\it Step 3}, we will prove that $c_1>0$; the expression for $c_2$ via $c_1$ 
is so far unknown and will be obtained at the end of the proof.

\textit{Step 2. (ii) $\to$ (i)}. This implication is straightforward. 

\textit{Step 3. (ii) $\to$ (iii)}.
 We compute $w(r)$ using formula \rf{inverse0}, by finding $xy^k$ and substituing
 $e^t$ by $r$. Noticing that $y = \frac{n-2k}{k}(1+o(1))$, we obtain
 the expression for $w(r)$ from {\it (iii)}. Furthermore, by \rf{newtrans0},
$w'(r) = -\frac{w(r)y(t)}{r} = -\frac{w}{r}\frac{n-2k}{k}(1+o(1))$ which implies the expression
for $w'(r)$ from item {\it (iii)}.
 
 Let us show now that the constant $c_1$ from {\it Step 1} is strictly positive. 
 From the expression for $w'$ in item {\it (iii)}, it follows that
 $r^{n-k} w'(r)^k = c_4^k (1+o(1))$. Furthermore, integrating
 \rf{Eq:IVP:0} gives
 \aa{
 r^{n-k} w'(r)^k = r^{n-k}_0 w'(r_0)^k + \int_{r_0}^r
  s^{n-1}c_{n,k}^{-1}\,\la \rho(s)(-w(s))^q ds >  r^{n-k}_0 w'(r_0)^k >0, 
  \;\text{where}\; r_0>0.
 }Passing to the limit as $r\to+\infty$, we obtain that $c_4^k>0$. But 
 $c_4 \sim c_1^{\frac1{q-k}}$, which implies that $c_1>0$.
 
 \textit{Step 4. (iii) $\to$ (i)}. Since $w'(r) = -\frac{w(r)y(t)}{r}$,
one obtains $y(t) = \frac{n-2k}{k}(1+o(1))$.
By \rf{inverse0},
  \aa{
  x(t) =  c_{n,k}^{-1} r^{2k} \rho(r) (-w)^{q-k} y(t)^{-k} = c_1 r^\gm(1+o(1)),
  \;\text{with}\; r=e^t.
  }
  The above representations for $x(t)$ and $y(t)$ imply {\it (i)}.
  
  \textit{Step 5. (i) $\to$ (iv).}   It follows from Lemma \ref{inward}
   that $\ffi(t)\in G_+$ for all $t\in \Rnu$. 
   
   \textit{Step 6. (iv) $\to$ (i).} 
   Note that $P_3^+$ and $P_4^+$ {are always} in $G_-$. {Indeed,
   we already proved in Corollary \ref{pro1111} that $P_4^+\in G_-$. Further,
   we note that $G_0$ intersects the $x$-axis at
   $(0, \frac{q+1}{k+1}(n-2k))$ and $n+l_\infty< \frac{q+1}{k+1}(n-2k)$ by
   \rf{inq*1111}.}
   Now it follows from Corollary \ref{pro1111} that only point $P_2$ can belong
   to the $\om$-limit set of the orbit associated to a solution of \rf{P2P3}.

It remains to show \rf{haty2222}. By Lemma \ref{lemp2222}, 
in a neighborhood of $P_2$, $y$ is a function of $x$ (denoted by $\hat y(x)$).
From \rf{p2xy}, we observe that 
\aa{
\lim_{x \to 0+} \frac{\hat y(x)-\frac{n-2k}k}x
= \lim_{t\to +\infty} \frac{y(t) - \frac{n-2k}k}{x(t)}
= \frac{c_2}{c_1}.
}
Formula \rf{haty2222} {holds} now by Lemma \ref{lemp2222}.
Furthermore, since $\hat y'(0)= \frac{c_2}{c_1}$, we find that
$c_2= c_1 \big(\frac{k^2\gm}{n-2k}- k\big)$. At the same time, finding $c_2$ concludes the proof of Step 1.
\end{proof}

\subsubsection{Examples of $P_2$-solutions}
Here we give examples of radial $P_2$-solutions to Problem
\rf{P2P3} for some explicitely given weight functions $\rho(|x|)$.
Our examples also contain explicit expressions for the solutions. 
Remark that in both examples,  the asymptotic 
representations at a neighborhood of $+\infty$, obtained in Theorem \ref{p2222},
agree with the explicit formulas for the respective solutions.

{\bf Example 1.}
Let $q>k$ and $n>2k$. Define 
\[
u(x)=w(r)=-(1+r^2)^{-\te},\, r=\abs{x},
\]
where $\te=\frac{n-2k}{2k}$. An easy calculation shows that
\[
S_k(D^2 w)=c_{n,k}\frac{(2\te)^k(n-k)}{1+r^2}(1+r^2)^{-k(\te+1)}.
\]
Now put $k(\te+1)=\te q$, so that $\te=\frac{k}{q-k}$. Therefore, 
\[
w(r)=-(1+r^2)^{-\frac{n-2k}{2k}}
\]
is a $P_2$-solution of 
$S_k(D^2 w)=\rho(\abs{x})(-w)^q$, $x\in\mathbb{R}^n$,
 with $\rho(\abs{x})=c_{n,k}\frac{(\frac{n-2k}{k})^k(n-k)}{1+\abs{x}^2}$. In this case, $R(r)=-\frac{2r^2}{1+r^2}$. Hence, $l_0=0$ and $l_\infty=-2$. 

 {\bf Example 2.}
Let $c>0$ any constant and $\sigma>0$. Consider the family of \lq\lq Bliss functions\rq\rq
\[
w_c(\abs{x})=-\frac{K}{(c+\abs{x}^\frac{2k+\sigma}{k})^\frac{n-2k}{2k+\sigma}},
\]
where  $K=\left[c_{n,k}(n+\sigma)\left(\frac{n-2k}{k}\right)^k c\right]^\frac{n-2k}{(2k+\sigma)(k+1)}$. Then $w_c$ is a $P_2$-solution of 
\[
S_k(D^2 w)=\abs{x}^\sigma(-w)^q,\;\; x\in\mathbb{R}^n,
\]
where $q=q^*(k,\sigma)=\frac{(n+2)k+\sigma (k+1)}{n-2k}$. In this case, we have $l_0=l_\infty=\sigma$. 

\subsection{$P_3^+$-solutions}
In this section, we characterize $P_3^+$-solutions.

\subsubsection{$P_3^+$-solutions of  algebraic fast decay: Proof of Theorem \ref{t1111}}
Here we consider the case $l_\infty<-2k$ or $\dl>0$, the necessary condition
when fast decay $P_3^+$-solutions may exist.
\begin{proof}[Proof of Theorem \ref{t1111}]
\textit{Step 1. (i) $\to$ (ii):}
 The expression $x(t) = \nu_+(1+o(1))$ follows immediately from the
 definition of a $P_3^+$-solution. Next, define
\aa{
z(t) = \frac{\nu_+ - x}{y}.
}
One can rewrite the ODE \rf{1111y} as follows:
\aa{
\Big(\frac1y\Big)' - \dl \, \frac1y - \frac{z}{k} + 1 = 0,
}
whose solution is
\aaa{
\lb{eq1111}
y(t) = \frac{e^{-\dl t}}{e^{-\dl t_0} y(t_0)^{-1} + \int_{t_0}^t e^{-\dl s}\big(\frac{z(s)}{k} - 1\big) ds}, \quad t_0\in \Rnu.
}
Let us determine the behavior of the term $e^{-\dl t} z(t)$. Differentiating $z(t)$
and taking into account that $\nu_+ - x = zy$, we obtain
\mmm{
\lb{1111z}
\dot z = -\frac{\dot x}{y} - z\,\frac{\dot y}y = -\frac{x}{y}(yz + \zeta(t)  - qy) - z\Big(-\dl + \frac{x-\nu_+}k + y\Big)\\
=\Big(\dl +\frac{\nu_+ -x}k - y - x\Big) z + qx  - \frac{x}{y} \, \zeta(t).
}
Next, for a differentiable function $\ffi$, we have
\aaa{
\lb{zffi1111}
(z\ffi)' = z' \ffi + \ffi' z = \Big(\dl+\frac{\nu_+-x}k - y - x+ \frac{\ffi'}{\ffi}\Big) z\ffi
+ q x\ffi - \frac{x}{y} \, \zeta \,\ffi.
}
Taking $\ffi(t) = e^{-(\dl- \eps) t}$, we obtain 
\aa{
(ze^{-(\dl-\eps) t})' = \Big(\frac{\nu_+-x}k - y - x + \eps\Big)ze^{-(\dl-\eps) t} + q\,x\,e^{-(\dl-\eps) t} +  \frac{x e^{-\dl t}}{y} e^{\eps t} \zeta(t).
}
Choose $\eps\in (0,\min\{\frac{\teta}2,\nu_+,\dl\})$. With this choice,
$e^{2\eps t} \zeta(t) \to 0$ as $t\to +\infty$ and $\lim_{t\to+\infty}(\frac{\nu_+-x}k - y - x + \eps)= -\nu_+ + \eps < 0$. We need to analyze the factor $\frac{e^{-\dl t}}{y}$ appearing in the last term.
Equation \rf{1111y} implies
\aa{
\Big(\frac{e^{-(\dl+ \eps)t}}{y} \Big)' = \frac{e^{-(\dl+ \eps)t}}y  \Big(\frac{\nu_+-x}k - \eps\Big) -  e^{-(\dl+ \eps)t}.
}
By Lemma \ref{lem1111}, $\lim_{t\to+\infty} \frac{e^{-(\dl+ \eps)t}}{y} = 0$. Therefore, 
for all $\eps>0$,
\aaa{
\lb{ay111}
\frac{e^{-\dl t}}{y} = e^{\eps t} o(1).
}
By Lemma \ref{lem1111},
$\lim_{t\to+\infty} ze^{-(\dl-\eps) t} = 0$. This implies that as $t\to +\infty$,
\aa{
|z|e^{-\dl t} = e^{-\eps t} o(1).
}
Therefore, the integral in the denominator of \rf{eq1111} converges. 
There are two choices: 
\aaa{
\lb{a}  \tag{$a$}
 &e^{-\dl t_0} y(t_0)^{-1} + \int_{t_0}^{+\infty} e^{-\dl s}\big(\frac{z(s)}{k} - 1\big) ds \ne 0 \quad \text{for some} \; t_0\in \Rnu\\
 \lb{b} \tag{$b$}
 &e^{-\dl t} y(t)^{-1} + \int_{t}^{+\infty} e^{-\dl s}\big(\frac{z(s)}{k} - 1\big) ds = 0, \quad   \text{for all} \; t\in\Rnu.
}
If \rf{a} holds, we obtained \rf{yassi}. Remark that the constant $c$ is positive by  \rf{newtrans0}.
In the case \rf{b}, 
\aaa{
\lb{est1111}
\Big|\frac{1}{y e^{\dl t}}\Big| \lt \frac1{k} \int_t^{+\infty} |z| e^{-\dl s} ds + \int_t^{+\infty} e^{-\dl s} ds \lt K_\eps e^{-\eps t}
}
for some constant $K_\eps>0$, depending on $\eps$. On the other hand, \rf{lv1111} implies that for any $\hat \eps>0$,
\aa{
(y e^{(\dl-\hat \eps)t})' = \Big(\frac{x-\nu_+}k + y - \hat \eps\Big)y e^{(\dl-\hat \eps)t}. 
}
By Lemma \ref{lem1111}, $\lim_{t\to+\infty} y e^{(\dl-\hat\eps)t} = 0$.  Therefore, for all $\hat \eps>0$, 
\aa{
y e^{\dl t} = e^{\hat \eps t} o(1).
}
Inequality \rf{est1111} can be written as $y e^{\dl t} O(1) = e^{\eps t}$.
Substituting there $y e^{\dl t}$ from the previous identity, we obtain
\aa{
e^{(\eps - \hat\eps)t}= o(1) \quad \text{for all} \;\; \eps, \hat\eps>0 
\;\; \text{sufficiently small.}
}
Clearly, the above identity cannot hold for all (even sufficiently small) $\eps, \hat\eps>0$. 
This implies that the case \rf{b} cannot be realized. 
Thus, we obtained {\it (ii)}.

\textit{Step 2. (ii) $\to$ (i)}. This implication is straighforward. 

\textit{Step 3. (ii) $\to$ (iii).} We apply formula \rf{inverse0}.
 First, we note that 
 $y^k x = \nu_+c^ke^{-\dl kt}(1+o(1))$.
Substituting this expression into \rf{inverse0}, by Lemma \ref{krho}, we obtain
\aa{
w(r) = - c_1\big(\rho(r)r^{2k+\dl k}\big)^{-\frac1{q-k}} (1+o(1)) = -c_1(1+o(1)),
}
where $c_1 = \big(\frac{\nu_+c_{n,k} c^k}{ \, c_\rho}\big)^{\frac1{q-k}}$. Above, we used
the identity
$2k+\dl k = -l_\infty$ following from the definition of $\dl$. 
Next, by \rf{newtrans0},  $w'(r) = -w(r) y(\ln r) r^{-1}$ which immediately implies the second expression in item {\it (iii)}. Integrating the expression for $w'$ from $r$ to $+\infty$
and taking into account that $w(+\infty) = -c_1$, we obtain the expression for $w(r)$
in  {\it (iii)}.

\textit{Step 4. (iii) $\to$  (i)}.  It is straightforward to obtain \rf{yassi} 
by the second formula in \rf{newtrans0} and the expressions for $w$ and $w'$ in {\it (iii)}.
Indeed,
$y(t) = r \frac{w'}{-w} = \frac{c_2}{c_1} e^{-\dl t}(1+o(1))$. 
Furthermore, by the first formula in \rf{newtrans0},
\aaa{
\lb{xt11112222}
x(t)= c_{n,k}^{-1}  \, r^k \rho(r)\Big(\frac{-w}{w'}\Big)^k (-w)^{q-k} = \nu_+(1+o(1)).
} 
Thus, we proved that $(x(t),y(t))$ tends to $P_3^+$ as $t\to +\infty$.

\textit{Step 5. (i) $\to$  (v)}.  Note that $\lim_{t\to+\infty} W(x(t),y(t)) = 
-\frac{n-2k}{k} + \frac{n+l_\infty}{k} < 0$ which immediately implies {\it (v)}.

\textit{Step 6. (i) $\to$  (iv)}. Note that $W(x,y)<0$ is equivalent to 
$\frac{k}{n-2k} y - 1 < -\frac{x}{n-2k}$. The latter implies that $G(x,y)<0$. Therefore,
$W_-\sub G_-$.

 \textit{Step 7. (iv) $\to$  (i)}. 
 By Corollary \ref{pro1111},  the $\om$-limit set of a trajectory of \rf{lv1111}
 is either $P_2$ or $P_3^+$.  Suppose it is $P_2$.
 Then, item {\it (iv)} of Theorem \ref{p2222} contradicts to  item {\it (iv)} 
 of the current theorem. Therefore, $\ffi(t)$ is a $P_3$-solution.

 \textit{Step 8. (v) $\to$  (i)}. This follows from the inclusion $W_-\sub G_-$ 
shown above.

Now we let $(\rho.6)$ be in force and let us obtain \rf{x1111}.

\textit{Step 9: Asymptotic  representation for $x$ under $(\rho. 6)$-(1).}
Recall that $\cpp= \lim_{t\to+\infty} e^{\dl t}\zeta(t)$.
Since $\nu_+>\dl$, from \rf{1111z}, by Lemma \ref{lem1111}, we obtain
that $\lim_{t\to+\infty} z(t) = \frac{(q-\sfrac{\cpp}{c})\nu_+}{\nu_+-\dl}$. Therefore,
\aa{
\nu_+ - x = yz = \frac{(qc-\cpp)\nu_+}{\nu_+ - \dl} e^{-\dl t} (1+o(1)).
}
which is the same as the first expression in \rf{x1111}.

\textit{Step 10: Asymptotic  representation for $x$ under $(\rho. 6)$-(2).}
Applying \rf{zffi1111} with $\ffi(t) = \frac1{e^{\dl t}\zeta(t)}$
and introducing $\td z = \frac{z}{e^{\dl t}\zeta}$, we obtain 
\aa{
\td z' = \Big(\frac{\nu_+-x}k - y - x - \frac{\zeta'}{\zeta}\Big) \td z 
+ \frac{qx}{e^{\dl t}\zeta} - \frac{x(1+o(1))}{c}.
}

By Lemma \ref{lem1111},
 $\lim_{t\to+\infty} \td z = -\frac{\nu_+}{c(\nu_+ - \hat \nu)}$.
This implies that 
\aa{
& z(t) = - e^{\dl t} \zeta(t)  \frac{\nu_+}{c(\nu_+ - \hat \nu)} (1+o(1)),\\
&\nu_+ - x = zy =  -\frac{\nu_+}{\nu_+ - \hat \nu} \zeta(t)(1+o(1)).
}Hence, the second formula in \rf{x1111} holds.

 \textit{Step 11: Obtaining \rf{haty1111}.}
First, we note that under $(\rho. 6)$-(1), $x$ can be represented as a function of $y$
which we denote by $\hat x(y)$. Indeed,
it follows from \rf{lv1111} that in a neighborhood of $+\infty$, $\dot y<0$, 
and hence, one can express $t$ as a function of $y$. 
Let us show now that under $(\rho. 6)$-(2), in a neighborhood of $+\infty$,
$y$ can be represented as a function of $x$.
Indeed,  $e^{-\dl t} = o(\zeta(t))$ and
 $\dot x = x\big(-\frac{\hat \nu}{\nu_+ - \hat \nu}\zeta(t) +o(\zeta(t))\big)$.
Therefore, in a neighborhood of $+\infty$,  $\dot x < 0$ and we can express
$t$ as a function of $x$. 
Finally, we compute $\hat x'(0)$ as
$\lim_{y\to 0+} \frac{\hat x(y)-\nu_+}{y} = \lim_{t\to+\infty}\frac{x(t) - \nu_+}{y(t)}$ and 
 $\hat y'(\nu_+)$ as
$\lim_{x\to \nu_+} \frac{\hat y(x)}{x-\nu_+} = \lim_{t\to+\infty}\frac{y(t)}{x(t) - \nu_+}$ by 
 \rf{yassi} and \rf{x1111}. These computations imply \rf{haty1111}.
\end{proof}

\subsubsection{$P_3^+$-solutions of slow log-like decay: Proof of Theorem 
\ref{bl1111}}
Recall that the necessary condition for slow-decay $P_3^+$-solutions to exist
is $l_\infty = -2k$ (or $\dl =0$).
\begin{proof}[Proof of Theorem \ref{bl1111}]
{\it Step 1. (i) $\to$  (ii).}
Formula \rf{eq1111} with $\dl=0$ implies
\aaa{
\lb{yy1111}
y = \frac{\frac1t}{\frac1{t\, y(t_0)} + \frac1t \int_{t_0}^t \big(\frac{z(s)}{k} - 1\big) ds}.
}
Let us derive an asymptotic representation for $z(t)$ by formula \rf{1111z}.
By \rf{ay111}, for all $\eps>0$, $\frac1{y} = e^{\eps t} o(1)$.
Note that the condition $k<\frac{n}2$ is equivalent to the condition $\nu_+>0$.
From \rf{1111z} and Lemma \ref{lem1111}, it follows that
\aa{
z(t) = q(1+o(1)).
}
Therefore, in \rf{yy1111}, the integral diverges. By L'Hopital's rule,
\aa{
y(t) = \frac1t \frac{k}{q-k}(1+o(1)).
}
Finally, 
$\nu_+ - x = n - 2k -x = zy = \frac{1}t \frac{q k}{q-k}(1+o(1))$.

{\it Step 2. (ii) $\to$  (i).} This implication is straightforward. 

{\it Step 3.  (ii) $\to$  (iii).} Note that
$x(1+o(1)) = (n-2k)(1+o(1))$ and $y^k x = \frac1{t^k}\big(\frac{k}{q-k}\big)^k(n-2k)(1+o(1))$.
This immediately implies the expresssion for $w(r)$ in {\it (iii)}. The second expression
in {\it (iii)} is implied by the formula $w'(r) = - w(r) y(\ln r) r^{-1}$.

{\it Step 4.  (iii) $\to$  (i).}
It is straightforward to obtain the expression for $y$, given in item {\it (ii)},
by the second formula in \rf{newtrans0} and the expressions for $w$ and $w'$ in {\it (iii)}.
Furthermore, by \rf{xt11112222},  $x(t) =  \nu_+(1+o(1))$. This implies that
$(x(t),y(t))\to P_3^+$ as $t\to +\infty$.
 
{\it Step 5.  (ii) $\to$  (v)}. Note that 
$W_- = \big\{y+\frac{x}k < \frac{n-2k}k\big\}$.
Representation {\it (ii)} implies
\aa{
y+\frac{x}k = \frac{n-2k}k -\frac{q}{q-k} \frac1t(1+o(1)) + \frac{k}{q-k} \frac1t(1+o(1))
=  \frac{n-2k}k - \frac1t(1+o(1))
}
which immediately implies {\it (v)}.

{\it Step 6.  (v) $\to$  (iv)}. This implication is straighforward since $W_-\sub G_-$.

{\it Step 7.  (iv) $\to$  (i)}. The argument  of {\it Step 7}
of the proof of Theorem \ref{t1111} works here without changes.

It remains to prove \rf{yn2k1111}.
The second equation in \rf{LVSrho}, along with the representations for
$x(t)$ and $y(t)$ from item {\it (ii)}, implies that $\dot y = -y(\frac1{t} + o(\frac1{t}))$.
Therefore, $\dot y<0$ at some neighborhood of $+\infty$ and we can express
$t$ as a function of $y$. This implies that $x$ is a function of $y$, which we denote
by $\hat x(y)$.
By using the representations for $x(t)$ and $y(t)$, we obtain
\aa{
\hat x'(0) = \lim_{y\to 0} \frac{\hat x(y)-(n-2k)}{y}
= \lim_{t\to +\infty} \frac{x(t) - (n-2k)}{y(t)} = - q.
}
From here and item  {\it (ii)}, we obtain the existence of the limit 
$\lim_{y\to 0} \hat x'(y) = \lim_{t\to+\infty} \frac{\dot x}{\dot y} = - q = \hat x'(0)$. 
By the continuity of $\hat x'(y)$ in $0$, we obtain that $\hat x(y)$ is decreasing in a neighborhood of $0$; therefore, $y$ is a function of $x$ in a neighborhood of $n-2k$.
As before, we denote this function by $\hat y(x)$. Formula
\rf{yn2k1111} follows now from the above expression for $\hat x'(0)$.
\end{proof}

\subsection{$P_4^+$-solutions}
It was established in Corollary \ref{pro1111} that $P_4^+$-solutions
can exist only if $l_\infty>-2k$ or, equivalently, $\dl<0$. 
Below, we prove Theorem \ref{7777} characterizing $P_4^+$-solutions.

\subsubsection{Proof of Theorem \ref{7777}}
\begin{proof}
{\it Step 1: (i) $\to$ (ii).} Since $x(t) = \td x (1+o(1))$ and $y(t) = \td y(1+o(1))$, 
item {\it (ii)} follows from \rf{newtrans0},  \rf{inverse0},  and Lemma \ref{krho}.

{\it Step 2: (ii) $\to$ (i)}. From \rf{newtrans0} and  \rf{inverse0},  it follows that
$x(t) = \td x (1+o(1))$ and $y(t) = \td y (1+o(1))$, which implies {\it (i)}.

{\it Step 3: (i) $\to$ (iii).}  As it was shown in Corollary \ref{pro1111},
$P_4^+\in G_-$ and is located
in the first quadrant (since $l_\infty>-2k$). This proves {\it (iii)}.

{\it Step 4. (iii) $\to$ (i).} 
By Corollary \ref{pro1111}, $w(r)$ can be either $P_2$ or $P_4^+$ solution. 
By Theorem \ref{p2222},  if $w(r)$ is a $P_2$-solution, then $\ffi(t) \in G_+$ for all $t$.
Therefore, $w(r)$ is a $P_4^+$-solution.

\textit{Step 5: (i) $\to$ (iv) under condition \rf{d1111}}. 
By the results of Subsection \ref{lineariz1111}, condition \rf{d1111} 
implies that the roots of the characteristic equation of the linearized
system \rf{limit1111}, denoted by $\la_1$ and $\la_2$, are real, negative, and different, so $P_4^+$ is a stable node. Let $\la_1>\la_2$.
In a  neighborhood of $P_4^+$, one has the representation 
of \rf{LVSrho} as a time-dependent perturbation of  \rf{limit1111}
given by \rf{perturba} with $\bar x = x-\td x$, $\bar y = y-\td y$, where
$(\td x,\td y)$ is given by \rf{interiorcritipoint:plus}
and the matrix $A$ given by \rf{a4444}.
As in the case of $P_2$-solutions, we
rewrite \rf{perturba} with respect to $\psi(t) = (\bar x(t),\bar y(t))$ as follows:
\aa{
\dot \psi(t) = A\psi(t) + \dlt(\psi(t)) + \eta(t,\psi(t)),
}
where $\dlt(\psi(t))$ and $\eta(t,\psi(t))$ denote the two last terms (respectively) 
on the right-hand side.

Let $v_1$ and $v_2$ be the unit eigenvectors corresponding to $\la_1$ and $\la_2$,
respectively. Since the roots 
are real and different,  $v_1$ and $v_2$ are non-colinear. 
Therefore, each vector $z\in\Rnu^2$ is a linear combination $z= z_1 v_1 + z_2 v_2$.
Define the projection operators $P_1z = z_1v_1$ and $P_2 z = z_2v_2$.
By the continuity of $P_1$ and $P_2$, there
exists a constant $C>0$ such that
\aa{
 |P_i z| \lt C |z|, \quad i=1,2.
}
In the same manner as in the proof of Theorem \ref{p2222}, one can write
equation \rf{psi2222}. Setting $\tau = \bar \tau$, we transform \rf{psi2222} as follows:
\mmm{
\lb{ij2222}
e^{-\la_1 t}\psi(t) = (e^{-\la_1\tau}P_1 + e^{-\la_2\tau}P_2)\psi(\tau) +
 e^{(\la_2-\la_1) t}P_2 \psi(\tau)
+ \int_{\tau}^t e^{-\la_1s}P_1 \big(\dlt(\psi(s))+  \eta(s,\psi(s))\big) ds\\
+ e^{-\la_1 t} \int_{\tau}^t e^{\la_2(t-s)}P_2 \big(\dlt(\psi(s))+  \eta(s,\psi(s))\big) ds
= (e^{-\la_1\tau}P_1 + e^{-\la_2\tau}P_2)\psi(\tau) \\
 + e^{(\la_2-\la_1) t}P_2 \psi(\tau)  + I(t) + J(t).
}
Note that $\lim_{t\to+\infty} e^{(\la_2-\la_1) t}P_2 \psi(\tau)  = 0$.
Let us show that $I(t)$ and $J(t)$ converge to finite limits.  
To this end, we find an exponentially decaying 
estimate for $|\psi(t)|$ by using Lemma 7.1 from \cite{BattLi}. First, we find $\sg>0$
such that 
\aaa{
\lb{beta1111-}
\beta: = \frac{C\sg}{|\la_1|} < 1.
}
Estimate \rf{estimate1111} can be obtained likewise. 
This estimate and equation \rf{ij2222} imply that there exist  a constant $K_1>0$
such that
\aa{
|\psi(t)| \lt K_1 e^{\la_1 t} + C\sg \int_\tau^t e^{\la_1 (t-s)}|\psi(s)| ds. 
}
By Lemma 7.1 from \cite{BattLi}, there exist a constant $K_2>0$ such that
$|\psi(t)| \lt K_2 e^{(\frac{C\sg}{1-\beta} + \la_1)t}$. 
Furthermore, estimates \rf{est2222} can be obtained likewise; namely,
$|\dlt(\psi(t))| \lt (1+q) |\psi(t)|^2$ and  $|\eta(t,\psi(t))| \lt K_3e^{-\teta t} |\psi(t)|$
for some constant $K_3>0$. Therefore,
\aaa{
& |I(t)| \lt K_4 \int_\tau^t \big(e^{(\frac{2C\sg}{1-\beta} + \la_1)s} + e^{(\frac{C\sg}{1-\beta} - \teta)s}\big) ds,\nonumber \\
& |J(t)| \lt K_5 \, e^{(\la_2-\la_1) t} \int_{\tau}^t 
 \big(e^{(\frac{2C\sg}{1-\beta} + 2\la_1-\la_2)s} 
 + e^{(\frac{C\sg}{1-\beta} - \teta+\la_1-\la_2)s}\big) ds,
 \lb{j1111}
}
where $K_4,K_5>0$ are constants.
In \rf{beta1111-}, we can take $\sg>0$ even smaller so that $\frac{2C\sg}{1-\beta} + \la_1<0$
and $\frac{C\sg}{1-\beta} - \teta<0$. With this choice of $\sg$,
the integral $I(+\infty)$ converges, so $\lim_{t\to+\infty} I(t)$ is finite. 
Furthermore, integrating in the right-hand side of \rf{j1111} and taking into account the factor
$e^{(\la_2-\la_1) t}$, we conclude that $\lim_{t\to+\infty} J(t) = 0$. 
Therefore, $\lim_{t\to+\infty} e^{-\la_1 t} \psi(t)$ exists.
We denote it by $\vec{c_1}{c_2}$. 

\textit{Step 6: (iv) $\to$ (i) under condition \rf{d1111}.} This implication is straightforward. 
%
\end{proof}

\subsection{Weights satisfying the entire set of assumptions $(\rho.1)$--$(\rho.6)$}
 \lb{examples2222}
Example 1 below describes a construction of the weight $\rho$ based
on the choice of the parameters $\dl$ and $\teta$ and 
the function $\psi$ with values in a closed subinterval
of $(0,+\infty)$.

\subsubsection{Example 1: General construction of the weight function $\rho(r)$.} 
Here we describe how to choose the function $R(r)$, and then give a formula 
explicitely defining $\rho(r)$ via $R(r)$, so that assumptions
 ($\rho.1$)--($\rho.3$) and $(\rho. 4)$--$(\rho. 6)$ are satisfied.

Let $R(r)$ be a continuous function on $(0,+\infty)$ such that
$l_0 = \lim_{r\to 0} R(r)$ exists and $l_0\gt R(r)$ for all $r>0$.
We specifiy properties of $R(r)$
at neighborhoods of $+\infty$ and $0$. We note that these additional
properties in a neighborhood of $0$ are needed only if we want to satisfy $(\rho.3)$;
in the current section on $P_2$-, $P_3^+$\!-, and $P_4^+$\!-solutions, we
do not use $(\rho.3)$.
Outside of neighborhoods of $0$ and $+\infty$, the behavior of $R(r)$
can be arbitrary.

First of all, choose parameters $\dl\in \Rnu$ and $\teta>0$ satisfying
\aa{
\dl < \min\Big\{ \frac{n-2k}{k+1}, \teta\Big\} \;\; \text{or} \;\;
\teta < \min\{ \dl, n-2k - \dl k\} \;\; \text{or} \;\; \dl = \teta <  \frac{n-2k}{k+1}.
}
The latter is equivallent to the condition $\nu_+>\min\{\dl,\teta\}$ from $(\rho. 6')$.
Define $l_\infty = -\dl k - 2k$ and choose $l_0\gt 0$ in such a way 
that $q > q^*(k,l_0)$.

\textit{Construction of $R(r)$ at a neighborhood of $+\infty$.}
Choose $\psi(r)$ (in a neighborhood of $+\infty$) taking values in a closed subinterval
of $(0,+\infty)$. 
Define $R(r) = r^{-\teta} \psi(r) + l_\infty$. 
If $\teta = \dl$, we also assume the existence of a finite limit $\lim_{r\to+\infty} \psi(r)$.

\textit{Construction of $R(r)$ at a neighborhood of $0$.}
This step is necessary only if we need to satisfy assumption $(\rho.3)$; otherwise,
we do not need to specify the behavior of $R(r)$ in a neighborhood of $0$.
Choose $R(r)$ in such a way that 
\aa{
\text{the integral} \; \int_0^r \frac{R(s) - l_0}{s}\, ds \;\;\;  \text{converges.}
}
Remark that the convergence of the above integral implies that 
$\lim_{r\to 0+} R(r) = l_0$.
For the function $K(r)=r^{-l_0}\rho(r)$ we have
\aa{
\ln K(r) = \ln K(\epsilon) + \int_{\epsilon}^r  \frac{R(t) - l_0}{t} dt.
}
By the above choice of $R(\fdot)$, $\lim_{\epsilon\to 0+} K(\epsilon)>0$ exists;
it will be referred to below as $K(0)$. 

\textit{Formulas for $\rho(r)$.} There are several
ways to compute $\rho(r)$ if we know $R(r)$. Indeed, the definition of $R(\fdot)$
is equivalent to $(\ln \rho(r))' = \frac{R(r)}{r}$. 
If we need to satisfy $(\rho.3)$, then $\rho(r)$
can be computed by formula \rf{rho1111}. If we do not need to satisfy $(\rho.3)$,
then we use Lemma \ref{krho}. Introduce the function $\mc K(r) = \rho(r) r^{-l_\infty}$.
One easily verifies that $(\ln \mc K(r))' = \frac{R(r) - l_\infty}{r}$. 
By $(\rho.5)$, $\int_r^{+\infty} \frac{R(s) - l_\infty}{s} ds \to 0$ as $r\to +\infty$.
Integrating the aforementioned identity from $r$ to $+\infty$ and noticing that,
by Lemma \ref{krho}, $\mc K(+\infty) = c_\rho$, we obtain the formula
for computing $\rho(r)$:
\aa{
\rho(r) = c_\rho r^{l_\infty} \exp\Big\{\int_r^{+\infty} \frac{l_\infty - R(s)}{s} ds \Big\}.
}

Let us verify the assumptions ($\rho.1$)--($\rho.3$) and $(\rho. 4)$--$(\rho. 6)$.
First, we note that ($\rho.1$), ($\rho.2$) and $(\rho. 4)$, $(\rho. 5)$
are clearly satisfied. Assumption $(\rho. 6)$ is fulfilled  by Corollaries \ref{cor1111} and \ref{cor2222}. Finally, if we are interested in weights satisfying $(\rho.3)$,
we notice that by \rf{rho1111} and since $R(r)\lt l_0$ for all $r >0$,
$K(r) \lt K(0)$.

\subsubsection{Example 2: Concrete examples of $\rho(r)$.} 
Define $\rho(r) = \frac{\al r^\beta}{\td \al + r^\gm}$, where
$\al,\td \al, \beta, \gm$ are postitive constants. In this case,
$R(r) = \beta - \frac{\gm r^\gm}{\td \al + r^\gm}$ so that $l_0 = \beta$
and $l_\infty = \beta - \gm$. Furthermore, 
$R(r) - l_\infty = \frac{\gm r^{-\gm}}{r^{-\gm} + \td \al^{-1}}$, and hence, $\teta = \gm$
and $\psi(r) = \frac{\gm}{r^{-\gm} + \td \al^{-1}}$. By Corollary \ref{cor1111},
we have to choose $\beta$ and $\gm$ in agreement with the inequality
$n+\beta-\gm > \min\{\dl,\gm\} = \dl$ since $\dl = \frac{\gm-\beta}{k} - 2$.
Therefore, $\beta$ and $\gm$ have to satisfy
$\beta - \gm > -\frac{k(n+2)}{k+1}$. 
Another example is $\rho(r) = r^\sg$, where $\sg>0$ and such that
$q>q^*(k,\sg)$. In this case $R(r) = \sg$, so $(\rho.1)$--$(\rho.6)$
are obviously satisfied.

\subsection*{Appendix}
\subsubsection*{Derivation of the Lotka-\!Volterra system}

We transform the equation $c_{n,k}r^{1-n}(r^{n-k}(w_r)^k)_r=f(r,w+1)$, where
$f(r,w+1):=\la \rho(r)(-w)^q$,
 into a Lotka-Volterra-type system. For this, define $\Phi(r):=r^{n-k}(w_r)^k$, $g(r,w):=c_{n,k}^{-1}r^{n-1}f(r,w+1)$, and introduce the variables
\aaa{
\lb{lv3333}
x(t)=r^k\frac{c_{n,k}^{-1}f(r,w+1)}{(w_r)^k},\;\; y(t)=r\frac{w_r}{-w},\;\; t=\ln(r).
}
The change of variable \rf{lv3333} was first introduced by Milne \cite{milne30,milne32} for the case $k=1$ and $\rho\equiv 1$. Derivating $x(t)$ in \rf{lv3333}, we obtain
\begin{eqnarray*}
x'&=&\frac{dx}{dr}\frac{dr}{dt}=r\left(\frac{rg}{\Phi}\right)_r=\frac{r}{\Phi^2}\left [\Phi(r(g_r+g_ww_r)+g)-rg\Phi_r\right]\\
&=&\frac{rg_wrw_r}{\Phi}+\frac{r^2g_r+rg}{\Phi}-\frac{(rg)^2}{\Phi^2}\\
&=&-\frac{ug_wr}{\Phi}\frac{rw_r}{-w}+\frac{rg(r\frac{g_r}{r}+1)}{\Phi}-\left(\frac{rg}{\Phi}\right)^2\\
&=&-\frac{rg\,\frac{ug_w}{g}}{\Phi}\frac{rw_r}{-w}+\left(r\frac{g_r}{g}+1\right)\frac{rg}{\Phi}-\left(\frac{rg}{\Phi}\right)^2\\
&=&-w\frac{g_w}{g}xy+\left(r\frac{g_r}{g}+1\right)x-x^2\\
&=&x\left(\left(r\frac{g_r}{g}+1\right)-x-w\frac{g_w}{g}y\right)
= x(\nu(t) - x -qy),
\end{eqnarray*}
where the last equality follows from the definition of $g$. Next,
we note that $c_{n,k}^{-1}r^{n-1}f(r,w+1)=\Phi_r$. Furthermore,
\begin{eqnarray*}
\Phi_r&=&r^{n-2k}(k(-wy)^{k-1}(-wy'e^{-t}-w_ry))+(n-2k)r^{n-2k-1}(-wy)^k\\
&=&r^{n-2k}(-k\frac{u}{r}y'(-wy)^{k-1}-kw_ry(-wy)^{k-1})+(n-2k)r^{n-2k-1}(-wy)^k\\
&=&r^{n-2k}\left(\frac{k}{r}\frac{y'}{y}(-wy)^k-ky\left(-\frac{uy}{r}\right)(-wy)^{k-1}\right)+(n-2k)r^{n-2k-1}(-wy)^k\\
&=&r^{n-2k-1}(-wy)^k\left (k\frac{y'}{y}-ky+n-2k\right)\\
&=&r^{n-2k-1}(rw_r)^k\left (k\frac{y'}{y}-ky+n-2k\right).
\end{eqnarray*}
The definition of $x(t)$ by \rf{lv3333} and the last equality imply that
$k\frac{y'}{y}-ky+n-2k=x$, which is equivalent to 
$y'=y\left(-\frac{n-2k}{k}+\frac{x}{k}+y\right)$.

\subsection*{Acknowledgements} 
 J.M.~do~\'O acknowledges partial support  from 
CNPq through the grants 312340/2021-4  and 
429285/2016-7 and 
Para\'iba State Research Foundation (FAPESQ), grant no 3034/2021.
J. S\'anchez is supported in part by ANID Fondecyt {Grant No.} 1221928.
 E. Shamarova acknowledges partial support from 
Universidade Federal da Para\'iba
 (PROPESQ/PRPG/UFPB) through the grant 
PIA13631-2020 (public calls no 03/2020 and 06/2021).


\Addresses


\end{document}